\documentclass[11pt]{amsart}
\usepackage{amsmath,amssymb,latexsym,esint}
\usepackage{verbatim}
\usepackage[left=3.1cm,right=3.1cm,top=3.4cm,bottom=3.4cm]{geometry}

\usepackage{cite}

\usepackage{color,enumitem,graphicx}
\usepackage[colorlinks=true,urlcolor=blue, citecolor=red,linkcolor=blue,linktocpage,pdfpagelabels, bookmarksnumbered,bookmarksopen]{hyperref}

\usepackage[hyperpageref]{backref}
\usepackage[english]{babel}

\numberwithin{equation}{section}
\newtheorem{theorem}{Theorem}[section]
\newtheorem{lemma}[theorem]{Lemma}

\newtheorem{corollary}[theorem]{Corollary}
\newtheorem{proposition}[theorem]{Proposition}

\theoremstyle{definition}
\newtheorem{remark}[theorem]{Remark}
\newtheorem{definition}[theorem]{Definition}
\newtheorem*{plan}{Plan of the paper}
\newtheorem*{acknow}{Acknowledgements}

\title[{Extremals for the fractional Sobolev inequality}]{Optimal decay of extremals for the \\ fractional Sobolev inequality}

\author[L.\ Brasco]{Lorenzo Brasco}
\author[S.\ Mosconi]{Sunra Mosconi}
\author[M.\ Squassina]{Marco Squassina}

\address[L.\ Brasco]{Aix-Marseille Universit\'e, CNRS
	\newline\indent
	Centrale Marseille, I2M, UMR 7373, 39 Rue Fr\'ed\'eric Joliot Curie
	\newline\indent
	13453 Marseille, France}
\email{lorenzo.brasco@univ-amu.fr}

\address[S.\ Mosconi \& M.\ Squassina]{Dipartimento di Informatica
	\newline\indent
	Universit\`a degli Studi di Verona
	\newline\indent
	Verona, Italy}
\email{sunrajohannes.mosconi@univr.it}
\email{marco.squassina@univr.it}


\begin{document}
	
	\begin{abstract}
	We obtain the sharp asymptotic behavior at infinity of extremal 
	functions for the fractional critical Sobolev embedding.
	\end{abstract}

\maketitle

\begin{center}
	\begin{minipage}{11cm}
		\small
		\tableofcontents
	\end{minipage}
\end{center}

\medskip

\section{Introduction and main result}
\noindent
Let $N>p>1$.\ In two seminal papers, T.\,Aubin \cite{Au} and G.\,Talenti \cite{talenti} showed that the minimizers of the Sobolev quotient
\begin{equation}
\label{opt-local}
\mathcal{S}_{p}=\inf_{u\in D^{1,p}(\mathbb{R}^N)\setminus\{0\}} \frac{\|\nabla u\|^p_{L^p(\mathbb{R}^N)}}{\displaystyle\left(\int_{\mathbb{R}^N} |u|^\frac{N\,p}{N-p}\, dx\right)^\frac{N-p}{N}},
\end{equation}
are given by the family of functions 
\begin{equation}
\label{aubintalenti}
U_{t}(x)={\mathcal C}\,t^\frac{p-N}{p}\,U\left(\frac{x-x_0}{t}\right),
\qquad {\mathcal C}\in\mathbb{R}\setminus\{0\},\ t>0,\ x_0\in\mathbb{R}^N,
\end{equation}
where
\[
U(x)=\mathcal{C}_{N,p}\,\big(1+|x|^{\frac{p}{p-1}}\big)^{\frac{p-N}{p}},\qquad\qquad \mathcal{C}_{N,p}=\left(\int_{\mathbb{R}^N}\big(1+|x|^{\frac{p}{p-1}}\big)^{-N}\,dx\right)^\frac{N\,p}{p-N}.
\]
For the limit case $p=1$, the problem was investigated by H.\,Federer and W.\,H.\,Fleming
in \cite{federer} and by V.\,G.\,Maz’ya in \cite{mazia}.
\par
On one side, these results establish an enlightening  connection between the theory of Sobolev spaces and the theory of classical isoperimetric inequalities.
On the other side, they provide a very powerful tool for the study of second order partial differential equations involving nonlinearities reaching
the critical growth with respect to the Sobolev embedding.
In the case $p=2$, these classification results were formally derived by G.\,Rosen in \cite{Ro}.
\par
The variational problem \eqref{opt-local} is related to the following equation involving the $p-$Laplace operator $\Delta_p u={\rm div}(|\nabla u|^{p-2}\nabla u)$,
\begin{equation}
\label{pEnt-loc}
-\Delta_p u=|u|^{\frac{N\,p}{N-p}-2}\,u, \qquad \mbox{ in }\mathbb{R}^N. 
\end{equation}
In fact, a nontrivial problem is that of proving that the only {\em fixed sign} solutions of this equation are precisely given by \eqref{aubintalenti}, for a suitable choice of the constant.
\par
In the restricted class of {\em radially symmetric fixed sign} solutions to \eqref{pEnt-loc},
this was shown by M.\,Guedda and L.\,Veron in \cite{GueddaV}. Recently, in \cite[Corollary 1.3]{Vetois} for the case $1<p\leq 2\,N/(N+2)$,
in \cite[Theorem 1.2]{ClassDino} for the case $2\,N/(N+2)<p\leq 2$ and in \cite[Theorem 1.1]{sciunziCl} for the case $2<p<N$, it was proved that any positive weak solution to \eqref{pEnt-loc} is radially symmetric and radially decreasing about some point,
thus answering positively to the classification of constant sign solutions to \eqref{pEnt-loc}. 
\par
The result by Aubin and Talenti, as well as the previous results in the linear case $p=2$, strongly rely on the reduction of the problem
 to an {\em ordinary differential equation} which can be explicitly solved. We recall that more recently, the Aubin-Talenti result has been reproved in \cite[Theorem 2]{CENV} by means of very different techniques, based on Optimal Transport.
\vskip.2cm
\noindent
Let now $s\in (0,1)$, $p>1$ and $N>s\,p$. The goal of this paper is to provide information about the asymptotic behavior at infinity of
optimizers of the problem
\begin{equation}
\label{costanteintro}
\mathcal{S}_{p,s}:=\inf_{u\in D^{s,p}(\mathbb{R}^N)\setminus\{0\}} \frac{\displaystyle\int_{\mathbb{R}^{2N}}\frac{|u(x)-u(y)|^p}{|x-y|^{N+s\,p}}\,dx\,dy}{\displaystyle\left(\int_{\mathbb{R}^N} |u|^\frac{N\,p}{N-s\,p}\,dx\right)^\frac{N-s\,p}{N}},
\end{equation}
which is related to the {\it fractional} Sobolev embedding, see for example \cite[Theorem 1]{MS}. Here 
\[
D^{s,p}(\mathbb{R}^N)=\left\{u\in L^{{N\,p/(N-s\,p)}}(\mathbb{R}^N):
\int_{\mathbb{R}^{2N}}\frac{|u(x)-u(y)|^p}{|x-y|^{N+s\,p}}\,dx\,dy<\infty\right\}.
\]
In the limit case $p=1$, the sharp constant above has been determined in \cite[Theorem 4.1]{FS} 
(see also \cite[Theorem 4.10]{BraLinPar}). The relevant extremals are given by characteristic functions of balls, exactly as in the local case.
\par
Problem \eqref{costanteintro} for $p>1$ is now related to the study of the nonlocal integro-differential equation 
\begin{equation}
\label{quazione}
(-\Delta_p)^s u=|u|^{\frac{N\,p}{N-s\,p}-2}\,u,\qquad \mbox{ in }\mathbb{R}^N, 
\end{equation}
where, formally, the operator $(-\Delta_p)^s$ is defined on smooth functions as
\[
(-\Delta_p)^s u(x) = 2\, \lim_{\varepsilon \searrow 0} \int_{\mathbb{R}^N \setminus B_\varepsilon(x)} \frac{|u(x) - u(y)|^{p-2}\, (u(x) - u(y))}{|x - y|^{N+s\,p}}\, dy, \qquad x \in \mathbb{R}^N.
\]
This operator appears in some recent works like \cite{AMRT} and \cite{IN}. See also \cite{DKP2,ILPS,IMS,IMS2,KMS1} and the references therein for some existence and regularity results.
\par
In the Hilbertian case $p=2$, it is known by \cite[Theorem 1.1]{Cotsiolis} that the family of functions
\begin{equation}
\label{cotsiolis}
U_{t}(x)={\mathcal C}\,t^\frac{2\,s-N}{2}\,\left(1+\left(\frac{|x-x_0|}{t}\right)^2\right)^{\frac{2\,s-N}{2}},\qquad \mathcal{C}\in\mathbb{R}\setminus\{0\},\ t>0,\ x_0\in\mathbb{R}^N,
\end{equation}
is the only set of minimizers for the best Sobolev constant $\mathcal{S}_{2,s}$. More precisely, in \cite[Theorem 1.1]{Cotsiolis} it is proved that the family \eqref{cotsiolis} provides all the minimizers of the following problem
\[
\widetilde{\mathcal{S}}_{2,s}:=\inf_{u\in D^{s,2}(\mathbb{R}^N)\setminus\{0\}} \frac{\big\|(-\Delta)^{s/2}\,u\big\|_{L^2(\mathbb{R}^N)}^2}{\displaystyle\left(\int_{\mathbb{R}^N} |u|^\frac{2\,N}{N-2\,s}\,dx\right)^\frac{N-2\,s}{N}},
\]
where the $L^2$ norm of $(-\Delta)^{s/2} u$ is defined in terms of the Fourier transform. By using \cite[Proposition 3.6]{Guida}, one knows that
\[
\int_{\mathbb{R}^{2N}}\frac{|u(x)-u(y)|^2}{|x-y|^{N+2\,s}}\,dx\,dy=c\, \big\|(-\Delta)^{s/2}\,u\big\|_{L^2(\mathbb{R}^N)}^2,
\]
for some $c=c(N,s)>0$. This implies that \eqref{cotsiolis} are the only solutions of \eqref{costanteintro} as well.
\par
It is also known by \cite[Theorem 1]{CLO} that, for a suitable 
positive constant ${\mathcal C}={\mathcal C}(N,s)$, \eqref{cotsiolis} are the only positive solutions of 
\begin{equation}
\label{linear-crit}
(-\Delta)^s u=|u|^{\frac{2\,N}{N-2\,s}-2}\,u \qquad \mbox{ in }\mathbb{R}^N.
\end{equation}
The result in \cite{CLO} is based upon the full equivalence between the weak solutions to \eqref{linear-crit} and the integral formulation
\begin{equation}
\label{integral-f}
u(x)=\int_{\mathbb{R}^N}\frac{|u(y)|^{\frac{2\,N}{N-2\,s}-2}\,u}{|x-y|^{N-2\,s}}\,dy,\qquad u\in L^{\frac{2\,N}{N-2\,s}}(\mathbb{R}^N),
\end{equation}
on the validity of some Kelvin transform and on moving plane arguments applied to \eqref{integral-f}, in the spirit of \cite{lieb}. 
\vskip.2cm
Unfortunately, in the nonlocal and nonlinear case $p\neq 2$ there is {\em no} Kelvin transform and {\em no}
equivalent integral representation result. 
Furthermore, even restricting to the class of radially symmetric functions, establishing a classification result for the optimizers
of \eqref{costanteintro} seems very hard. 
We conjecture that 
the optimizers are given by
\begin{equation}
\label{pTal}
U_{t}(x)={\mathcal C}\,t^\frac{s\,p-N}{p}\,U\left(\frac{x-x_0}{t}\right),
\qquad {\mathcal C}\in\mathbb{R}\setminus\{0\},\ t>0,\ x_0\in\mathbb{R}^N,
\end{equation}
where this time
\begin{equation}
\label{famloc}
U(x):=\mathcal{C}_{N,p,s}\,\big(1+|x|^{\frac{p}{p-1}}\big)^{\frac{s\,p-N}{p}},\qquad\qquad \mathcal{C}_{N,p,s}=\left(\int_{\mathbb{R}^N}\big(1+|x|^{\frac{p}{p-1}}\big)^{-N}\,dx\right)^\frac{N\,p}{s\,p-N}.
\end{equation}
Notice that \eqref{pTal} and \eqref{famloc} are consistent with the cases $p=2$ or $s=1$, in the last case we are back to the family of Aubin-Talenti functions \eqref{aubintalenti}
for the $p-$Laplacian operator.
\par
In the main result of this paper, we prove that extremals for \eqref{costanteintro} have exactly the decay rate at infinity dictated by formula \eqref{famloc}. Namely, we have the following.
\begin{theorem}
\label{sharp-Dec}
Let $U\in D^{s,p}(\mathbb{R}^N)$ be any minimizer for \eqref{costanteintro}. 
Then $U\in L^\infty(\mathbb{R}^N)$ is a constant sign, radially symmetric and monotone function with
\begin{equation}
\label{sciarpa2}
\lim_{|x|\to\infty}|x|^{\frac{N-s\,p}{p-1}}U(x)=U_\infty,
\end{equation}
for some constant $U_\infty\in\mathbb{R}\setminus\{0\}$.
\end{theorem}
\begin{remark}
\label{referee}
As it will be apparent from the proof of Theorem \ref{sharp-Dec}, the same conclusion \eqref{sciarpa2} can be drawn for any constant sign, radially symmetric and monotone solution of the critical equation \eqref{quazione}. In the local case this property is a plain consequence of the aforementioned classification result of \cite{GueddaV}.
\end{remark}
The building blocks of Theorem~\ref{sharp-Dec} are a weak $L^q$ estimate for the minimizers (Proposition \ref{lp-weak}), a Radial Lemma for Lorentz spaces (Lemma \ref{lm:border}) and the fact that the function
$$
\Gamma(x):=|x|^{-\frac{N-s\,p}{p-1}},\quad x\in\mathbb{R}^N\setminus\{0\},
$$ 
is a weak solution of $(-\Delta_p)^su=0$ in $\mathbb{R}^N\setminus B_r$, for any $r>0$ (Theorem \ref{fundamental}). Then the crucial point will be constructing  
suitable barrier functions to be combined with a version of the comparison principle for $(-\Delta_p)^s$
recently obtained in \cite{IMS}. 
Observe that for $s=1$, the function $\Gamma$ above is nothing but the fundamental solution of the $p-$Laplacian.
\par 
We wish to stress that Theorem~\ref{sharp-Dec} also provides  a very
useful tool for the investigation of {\em existence} of weak solutions for the nonlocal Brezis-Nirenberg problem in 
a smooth bounded domain $\Omega\subset\mathbb{R}^N$, i.e.
\[
\begin{cases}
(-\Delta_p)^s\, u  =\lambda\,|u|^{p-2}\,u+|u|^{\frac{N\,p}{N-s\,p}-2}u & \mbox{ in } \Omega, \\
u  = 0 & \mbox{ in }\mathbb{R}^N \setminus \Omega,
\end{cases}
\]
where $\lambda> 0$. This problem has been studied in \cite{SV} for $p=2$. For a general exponent $1<p<N/s$, by means of \eqref{sciarpa2}, one can estimate truncations of $U_t$
via a suitable cut-off function in terms of the sharp
constant $\mathcal{S}_{p,s}$ {\em without knowing the explicit form} of the optimizers. Such a procedure is new even for the local case.
These estimates allow to apply mountain pass or linking arguments by forcing
the min-max levels to fall inside a compactness range for the energy functionals, see \cite{per-mosc-squ-yang} for more details.

\begin{plan}
In Section \ref{sec:2} we set all the notations, definitions and basic facts that will be needed throughout the paper. Then in Section \ref{sec:3} we prove existence of solutions for \eqref{costanteintro}, together with some basic integrability properties. We also prove that extremals have to be comparable to 
$$
x\mapsto |x|^{-\frac{N-s\,p}{p-1}},\quad x\in\mathbb{R}^N\setminus\{0\},
$$ 
at infinity (Corollary \ref{lower-Dec}). Then the exact behavior \eqref{sciarpa2} is proved in Section \ref{sec:4}. The paper ends with Appendix \ref{sec:appA}, containing a rigourous computation of the fractional $p-$Laplacian of a power function. 
\end{plan}

\begin{acknow}
We warmly thank Yannick Sire for some informal discussions on the subject of this paper. We owe Remark \ref{referee} to the kind courtesy of an anonymous referee, we wish to thank him.
This research has been partially supported by {\it Gruppo Nazionale per l'Analisi Matematica, la Probabilit\`a e le loro Applicazioni} (INdAM) and by {\it Agence Nationale de la Recherche}, through the project ANR-12-BS01-0014-01 {\sc Geometrya}. 
Part of this paper was written during a visit of S.\,M. and M.\,S. in Marseille in March 2015. The I2M and FRUMAM institutions are gratefully acknowledged. 
\end{acknow}

\section{Preliminary results}

\label{sec:2}

\subsection{Notation}
In the following we will fix $s\in (0, 1)$, $p>1$ and $N$ as the dimension, letting for brevity
\[
p^*=\frac{N\,p}{N-s\,p}.
\]
We denote by $\omega_N$ the measure of the $N-$dimensional ball having unit radius.
Moreover, $\mathbf{S}^{N-1}$ will denote $\{x\in \mathbb{R}^N:|x|=1\}$. For $E\subseteq \mathbb{R}^N$ measurable we denote by $|E|$ its $N-$dimensional Lebesgue measure, by $E^c=\mathbb{R}^N\setminus E$ its complement and by $\chi_E$ its characteristic function. If $u:E\to \mathbb{R}$ is measurable we set
\[
[u]_{W^{s, p}(E)}^p:=\int_{E\times E}\frac{|u(x)-u(y)|^p}{|x-y|^{N+s\,p}}\, dx\, dy, \qquad [u]_{s,p}:=[u]_{W^{s,p}(\mathbb{R}^N)}, 
\]
and for any $q>0$
\[
\|u\|_{L^q(E)}:=\left(\int_{E}|u|^q\, dx\right)^{1/q},\qquad\|u\|_q:=\|u\|_{L^q(\mathbb{R}^N)}.
\]
Finally, for $t\in \mathbb{R}$ we will use the notation
\[
J_p(t)=|t|^{p-2}\,t.
\]

\subsection{Elementary inequalities}

We list here some useful inequalities on the function $J_p$. First, 
consider the case $p\geq 2$. We recall that 
\begin{equation}
\label{licciz}
|J_p(a)-J_p(b)|\le (p-1)\,(|a|^{p-2}+|b|^{p-2})\,|a-b|,\qquad a,b\in\mathbb{R},\quad p\ge 2,
\end{equation}
as a consequence of the mean value Theorem.
In \cite[eq. (2.7)]{IMS} it is also proved the following inequality
\begin{equation}
\label{27}
J_p(a)-J_p(a+b)\leq -2^{2-p}\,b^{p-1},\qquad a\in \mathbb{R}, \ b\geq 0,\quad p\geq 2.
\end{equation}
Let us consider the case $p\in (1,2]$. 
We recall the well-known monotonicity inequality
\begin{equation}
\label{seconda}
\big(J_p(a)-J_p(b)\big)\,(a-b)\ge c\,\frac{|a-b|^2}{(a^2+b^2)^\frac{2-p}{2}},\qquad a,b\in\mathbb{R}\setminus\{0\},\quad p\in(1,2].
\end{equation}
Next we prove the following inequality
\begin{equation}
\label{ineqb}
J_p(a)-J_p(a-b)\geq \max\left\{J_p(A)-J_p(A-b),\, \left(\frac{b}{2}\right)^{p-1}\right\},\quad a\in [0,A], b\geq 0,\ p\in (1,2].
\end{equation}
We distinguish two cases. First suppose that $a\geq b/2$. The function $t\mapsto J_p(t)-J_p(t-b)$ is readily seen to be decreasing on $[b/2, +\infty[$, so that 
\[
J_p(a)-J_p(a-b)\geq J_p(A)-J_p(A-b)
\]
in this case.
On the other hand, if $a<b/2$, being $J_p$ odd and increasing we have
\[
J_p(a)-J_p(a-b)\geq J_p(b-a)\geq J_p\left(\frac{b}{2}\right),
\]
and thus \eqref{ineqb} is proved.

\subsection{Functional framework}

We consider the space 
\[
D^{s,p}_0(\Omega):=\Big\{u\in L^{p^*}(\Omega): \text{$u\equiv 0$ in $\Omega^c$},\,\, [u]_{s,p}<+\infty\Big\},\qquad
D^{s,p}(\mathbb{R}^N):=D^{s,p}_0(\mathbb{R}^N),
\]
which is a Banach space with respect to the norm $[\,\cdot\, ]_{s,p}$.  Our first aim is to prove, under suitable regularity assumptions on $\partial\Omega$, that $C^\infty_c(\Omega)$ is dense in  $D^{s,p}_0(\Omega)$
with respect to the norm $[\,\cdot \, ]_{s,p}$. 
While this density result is well-known for $D^{s,p}_0(\Omega)\cap L^p(\Omega)$ (see for example \cite{FSV}), we will need to remove the $L^p$ assumption in the following. Finally we will prove a comparison principle in a rather general space.

\begin{theorem}
\label{density}
Let $\Omega\subset\mathbb{R}^N$ be an open set such that $\partial\Omega$ is compact and locally the graph of a continuous function.  Then $D^{s,p}_0(\Omega)$ is the completion of $C^\infty_c(\Omega)$ with respect to the norm $[\,\cdot\,]_{s,p}$.
\end{theorem}

\begin{proof}
Let $u\in D^{s,p}_0(\Omega)$. Reasoning on  $u_+$ and $u_-$ separately (which still belong to $D^{s,p}_0(\Omega)$), we can suppose that $u$ is nonnegative. Consider, for $\varepsilon>0$, the function $u_\varepsilon=(u-\varepsilon)_+$. Using the $1-$Lipschitzianity of $t\mapsto (t-\varepsilon)_+$ it is readily checked that
\[
|u_\varepsilon(x)-u_\varepsilon(y)|^p \leq |u(x)-u(y)|^p, \qquad |u_\varepsilon(x)-u_\varepsilon(y)|^p \to |u(x)-u(y)|^p,\ \text{ a.e. in $\mathbb{R}^{2N}$}.
\]
Therefore $u_\varepsilon\in D^{s,p}_0(\Omega)$ and by dominated convergence $[u_\varepsilon]_{s,p}\to [u]_{s,p}$. This in turn implies that $u_\varepsilon\to u$ in $D^{s,p}_0(\Omega)$ by uniform convexity of the norm. Now Chebyshev's inequality  ensures that $\mathrm{supp}(u_\varepsilon)$ has finite measure, thus by H\"older's inequality we get $u_\varepsilon\in L^p(\mathbb{R}^N)$. This yields 
\[
u_\varepsilon\in D^{s,p}_0(\Omega)\cap L^p(\mathbb{R}^N),
\]
and \cite[Theorem 6]{FSV} ensures that $u_\varepsilon$ can be approximated, 
in the norm $[\, \cdot\, ]_{s,p},$ by functions which belong to $C^\infty_c(\Omega)$.
\end{proof}

\noindent
We recall the following nonlocal Hardy inequality proved in  \cite[Theorem 2]{FS}.

\begin{proposition}[Hardy's inequality]
Let $N>s\,p$. Then there exists $C=C(N, p, s)>0$ such that 
\begin{equation}
\label{hardyin}
\int_{\mathbb{R}^N}\frac{|u|^p}{|x|^{s\,p}}\, dx\leq C\,[u]_{s, p}^p,\qquad \mbox{ for every } u\in D^{s,p}(\mathbb{R}^N).
\end{equation}
\end{proposition}

\noindent
We then define a suitable space where a comparison principle holds true. For any $\Omega\subset \mathbb{R}^N$ open set, we define 
\[
\begin{split}
\widetilde{D}^{s,p}(\Omega):=\Big\{u\in L^{p-1}_{\mathrm{\mathrm{loc}}}(\mathbb{R}^N)\cap L^{p^*}(\Omega) \,:\, &\exists\, E\supset\Omega \text{ with $E^c$ compact},\, \mathrm{dist}(E^c, \Omega)>0  \\ 
&\mbox{ and }\,[u]_{W^{s,p}(E)}<+\infty \Big\}.
\end{split}
\]
We wish to point out that the definition above is given having in mind the case of $\Omega$ being an exterior domain, i.e. the complement of a compact set.
Essentially, we consider functions $u$ which are regular in a slight enlargement
of $\Omega$ and possibly rough far from $\Omega$.
\vskip.2cm
The following expedient result will be used in the sequel.
\begin{lemma}[Nash-type interpolation inequality]
Let $1<p<\infty$ and $0<s<1$. For every $u\in L^{p-1}(B_R)$ such that $[u]_{W^{s,p}(B_R)}<+\infty$ we have
\begin{equation}
\label{disuguaglianza}
\|u\|^p_{L^p(B_R)}\le C\,R^{s\,p}\, [u]^p_{W^{s,p}(B_R)}+\frac{C}{R^\frac{N}{p-1}}\,\|u\|^p_{L^{p-1}(B_R)},
\end{equation}
for some $C=C(N,s,p)>0$.
\end{lemma}
\begin{proof}
We observe that it is enough to prove \eqref{disuguaglianza} for $R=1$, then the general case can be obtained with a simple scaling argument. 
\par
At first, we prove \eqref{disuguaglianza} for functions in $W^{s,p}(B_1)$. We can use a standard compactness argument: assume by contradiction that \eqref{disuguaglianza} is false on $W^{s,p}(B_1)$, then there exists a sequence $\{u_n\}_{n\in\mathbb{N}}\subset W^{s,p}(B_1)$ such that
\begin{equation}
\label{azzurdo}
\|u_n\|^p_{L^p(B_1)}=1\qquad \mbox{ and }\qquad [u_n]^p_{W^{s,p}(B_R)}+\|u_n\|^p_{L^{p-1}(B_1)}\le \frac{1}{n}.
\end{equation}
In particular, the sequence is bounded in $W^{s,p}(B_1)$. Thus by compactness of the embedding $W^{s,p}(B_1)\hookrightarrow L^{p}(B_1)$ (see for example \cite[Theorem 7.1]{Guida}) we get that (up to a subsequence) it converges strongly in $L^p(B_1)$ to $u\in W^{s,p}(B_1)$. From \eqref{azzurdo} we now easily get a contradiction. This shows that \eqref{disuguaglianza} is true for functions in $W^{s,p}(B_1)$.
\par
We now take $u\in L^{p-1}(B_1)$ with finite Gagliardo seminorm. Observe that
\begin{equation}
\label{valorassoluto}
\big||u(x)|-|u(y)|\big|\le |u(x)-u(y)|,
\end{equation}
so that
\[
\big[|u|\big]_{W^{s,p}(B_1)}\le [u]_{W^{s,p}(B_1)}.
\]
Thus we can assume $u$ to be positive without loss of generality. We define the increasing sequence $u_n=\min\{u,n\}\in W^{s,p}(B_1)$. From the first part of the proof and $1-$Lipschitzianity of the function $t\mapsto \min\{t,n\}$ we have
\[
\begin{split}
\|u_n\|^p_{L^p(B_1)}&\le C\,[u_n]^p_{W^{s,p}(B_1)}+C\,\|u_n\|^p_{L^{p-1}(B_1)}\\
&\le C\,[u]^p_{W^{s,p}(B_1)}+C\,\|u_n\|^p_{L^{p-1}(B_1)}.
\end{split}
\]
Passing to the limit and using the Monotone Convergence we get the desired conclusion.
\end{proof}
\begin{lemma}
\label{lm:frustone}
Let $1<p<\infty$ and $0<s<1$. For every $u\in L^{p-1}_{\mathrm{\mathrm{loc}}}(\mathbb{R}^N)$, every $E\subset \mathbb{R}^N$ open set and every ball $B_R\subset E$, we have
\begin{equation}
\label{frustone}
\int_{E} \frac{|u(x)|^{p}}{(1+|x|)^{N+s\,p}}\, dx\le C\,[u]^p_{W^{s,p}(E)}+C\,\|u\|^p_{L^{p-1}(B_R)},
\end{equation} 
for some $C=C(N,p,s,R)>0$, blowing-up as $R\searrow 0$.
\end{lemma}
\begin{proof}
We assume that the right-hand side on \eqref{frustone} is finite, otherwise there is nothing to prove. For simplicity, we can suppose that $B_R$ is centered at the origin. From \eqref{disuguaglianza}, we infer
\begin{equation}
\label{disuguaglianza2}
\int_{B_R} \frac{|u|^p}{(1+|x|)^{N+s\,p}}\,dx\le C\,R^{s\,p}\, [u]^p_{W^{s,p}(B_R)}+\frac{C}{R^\frac{N}{p-1}}\,\|u\|^p_{L^{p-1}(B_R)}.
\end{equation}
On the smaller ball $B_{R/2}$ (still centered at the origin), we have
\[
\int_{E\setminus B_R} \int_{B_{R/2}} \frac{|u(x)-u(y)|^p}{|x-y|^{N+s\,p}}\,dy\,dx\le [u]^p_{W^{s,p}(E)}<+\infty.
\]
Since
\[
|x-y|\le \frac{3}{2}\, |x|,\qquad x\in E\setminus B_R,\ y\in B_{R/2},
\]
we get
\[
\begin{split}
\int_{E\setminus B_R} \int_{B_{R/2}} \frac{|u(x)-u(y)|^p}{|x-y|^{N+s\,p}}\,dy\,dx&\ge c\,R^N\,\int_{E\setminus B_R} \frac{|u|^p}{|x|^{N+s\,p}}\,dx\\
&-c\, \left(\int_{E\setminus B_R} \frac{1}{|x|^{N+s\,p}}\,dx\right)\, \int_{B_{R/2}} |u|^p\,dy.
\end{split}
\]
In conclusion, the previous estimate proves
\[
\int_{E\setminus B_R} \frac{|u|^p}{(1+|x|)^{N+s\,p}}\,dx\le \frac{C}{R^N}\,[u]_{W^{s,p}(E)}^p+\frac{C}{R^{N+s\,p}}\, \int_{B_{R/2}} |u|^p\,dx.
\]
Using \eqref{disuguaglianza} to estimate the $L^p$ norm in the right-hand side gives
\begin{equation}
\label{disuguaglianza3}
\int_{E\setminus B_R} \frac{|u|^p}{(1+|x|)^{N+s\,p}}\,dx\le \frac{C}{R^N}\,\,[u]_{W^{s,p}(E)}^p+\frac{C}{R^{N\,\frac{p}{p-1}+s\,p}}\, \left(\int_{B_{R/2}} |u|^{p-1}\,dy\right)^\frac{p}{p-1},
\end{equation}
possibly for a different constant $C=C(N,s,p)>0$.
By summing up \eqref{disuguaglianza2} and \eqref{disuguaglianza3} we get the conclusion.
\end{proof}
\noindent
The next proposition shows that in the space $\widetilde{D}^{s, p}(\Omega)$, the operator $(-\Delta_p)^s$ is well defined.

\begin{proposition}
\label{deltaps}
For any $u\in \widetilde{D}^{s,p}(\Omega)$, the operator
\[
D^{s,p}_0(\Omega)\ni \varphi\mapsto  \langle (-\Delta_p)^s u, \varphi\rangle:=\int_{\mathbb{R}^N\times \mathbb{R}^N}\frac{J_p(u(x)-u(y))\,(\varphi(x)-\varphi(y))}{|x-y|^{N+s\,p}}\, dx\, dy
\]
is well defined and belongs to the dual space $(D^{s, p}_0(\Omega))^*$.
\end{proposition}

\begin{proof}
We proceed as in \cite[Lemma 2.3]{IMS}. Let $E\supset \Omega$ be such that $E^c$ is compact, ${\rm dist}(E^c, \Omega)>0$ and $[u]_{W^{s,p}(E)}<+\infty$.
Since $\varphi\equiv 0$ in $\Omega^c$, we split the integral as 
\[
\begin{split}
\int_{\mathbb{R}^N\times \mathbb{R}^N}&\frac{J_p(u(x)-u(y))\,(\varphi(x)-\varphi(y))}{|x-y|^{N+s\,p}}\, dx\, dy\\
&=\int_{E\times E}\frac{J_p(u(x)-u(y))\,(\varphi(x)-\varphi(y))}{|x-y|^{N+s\,p}}+2\int_{\Omega\times E^c}\frac{J_p(u(x)-u(y))\,\varphi(x)}{|x-y|^{N+s\,p}}\, dx\, dy.
\end{split}
\]
By H\"older's inequality the first term is finite and defines a continuous linear functional on $D^{s,p}_0(\Omega)$. Let us focus on the second one. By using that $\varphi\equiv 0$ in $E^c$, we need to show that 
\[
\varphi\mapsto \int_{\Omega}\varphi(x)\left(\int_{E^c} \frac{J_p(u(x)-u(y))}{|x-y|^{N+s\,p}}\, dy\right) dx,
\]
is a continuous linear functional on $D^{s,p}_0(\Omega)$. By means of Hardy's inequality \eqref{hardyin}, we get that convergence of $\{\varphi_n\}_{n\in\mathbb{N}}$ in $D^{s,p}_0(\Omega)\subset D^{s,p}(\mathbb{R}^N)$ implies strong convergence in $L^p(\Omega)$ of $\{|x|^{-s}\varphi_n\}_{n\in\mathbb{N}}$. Thus to prove the claim it suffices to show that 
\[
x\mapsto |x|^s\int_{E^c} \frac{J_p(u(x)-u(y))}{|x-y|^{N+s\,p}}\, dy\in L^{p'}(\Omega).
\]
Being $E^c$ compact and ${\rm dist}(E^c, \Omega)\geq \delta>0$ it holds
\begin{equation}
\label{xmy}
|x-y|\geq C\,(1+|x|), \qquad \mbox{ for every } x\in \Omega,\ y\in E^c,
\end{equation}
for some $C=C(E,\Omega)>0$. Thus, for almost every $x\in\Omega$, we can estimate
\[
\left|\int_{E^c} \frac{|x|^s\,J_p(u(x)-u(y))}{|x-y|^{N+sp}}\, dy\right|\leq C\left[|E^c|\,\frac{|u(x)|^{p-1}}{(1+|x|)^{\frac{N+sp}{p'}}}+\frac{1}{(1+|x|)^{N+s\,(p-1)}}\,\int_{E^c}|u|^{p-1}\, dy\right].
\]
The first term belongs to $L^{p'}(\Omega)$ due to \eqref{frustone}.
For the second one this follows from a direct computation. This proves the claim and the proposition.
\end{proof}

\begin{definition}
\label{defcomp}
Let $u\in \widetilde{D}^{s,p}(\Omega)$ and $\Lambda\in (D^{s,p}_0(\Omega))^*$. We say that $(-\Delta_p)^s u\leq \Lambda$ weakly in $\Omega$ if for all $\varphi\in D^{s,p}_0(\Omega)$, $\varphi\ge 0$ in $\Omega$,
\[
\int_{\mathbb{R}^{2N}}\frac{J_p(u(x)-u(y))\,(\varphi(x)-\varphi(y))}{|x-y|^{N+s\,p}}\,dx\,dy \le \langle\Lambda, \varphi\rangle.
\]
\end{definition}

\begin{theorem}[Comparison principle in general domains]
\label{comparison}
Let $\Omega\subset\mathbb{R}^N$ be an open set. Let $u, v\in \widetilde{D}^{s,p}(\Omega)$ satisfy 
\[
u\le v\ \mbox{ in }\Omega^c\qquad \mbox{ and }\qquad (-\Delta_p)^s u\leq (-\Delta_p)^s v\ \mbox{ in } \Omega.
\]
Then $u\le v$ in $\Omega$.
\end{theorem}

\begin{proof}
It suffices to proceed as in \cite[Lemma 9]{LL}, we only need to prove that $w:=(u-v)_+$ is an admissible test function, i.e. it belongs to $D^{s,p}_0(\Omega)$. Clearly $w\equiv 0$ in $ \Omega^c$ and $w\in L^{p^*}(\mathbb{R}^N)$. To estimate the Gagliardo seminorm, let $E\supset \Omega$ be such that $E^c$ is compact,  ${\rm dist}(E^c, \Omega)>0$ and 
\begin{equation}
\label{kl}
\int_{E\times E} \frac{|u(x)-u(y)|^p}{|x-y|^{N+s\,p}}\, dx\, dy+\int_{E\times E} \frac{|v(x)-v(y)|^p}{|x-y|^{N+s\,p}}\, dx\, dy<+\infty.
\end{equation}
Then
\[
\int_{\mathbb{R}^{2N}} \frac{|w(x)-w(y)|^p}{|x-y|^{N+s\,p}}\, dx\, dy=\int_{E\times E} \frac{|w(x)-w(y)|^p}{|x-y|^{N+s\,p}}\, dx\, dy+2\int_{\Omega\times E^c} \frac{|w(x)|^p}{|x-y|^{N+s\,p}}\, dx\, dy,
\]
and the first integral is finite due to 
\[
|w(x)-w(y)|^p\leq C\,(|u(x)-u(y)|^p+|v(x)-v(y)|^p),
\]
and \eqref{kl}. For the second one we use \eqref{xmy}, and since $|w(x)|^p\leq C(|u(x)|^p+|v(x)|^p)$ we get
\[
\int_{\Omega\times E^c} \frac{|w(x)|^p}{|x-y|^{N+sp}}\, dx\, dy\leq C\,|E^c| \int_{\Omega} \frac{|u(x)|^p}{(1+|x|)^{N+sp}}\, dx+ C\,|E^c|\int_{\Omega} \frac{|v(x)|^p}{(1+|x|)^{N+sp}}\, dx.
\]
The last two terms are finite, due the definition of $\widetilde{D}^{s,p}(\Omega)$ and \eqref{frustone}.
\end{proof}

\noindent
Finally, for the reader's convenience we recall the following result from \cite{IMS}. The proof is identical to the one of \cite[Lemma 2.8]{IMS} and we omit it.

\begin{proposition}[Non-local behavior of $(-\Delta_p)^s$]
\label{nonlocalb}
Let $N>s\,p$ and let $\Omega\subset\mathbb{R}^N$ be an open set such that $\partial\Omega$ is compact and locally the graph of continuous functions. Suppose that $u\in \widetilde{D}^{s,p}(\Omega)$ weakly solves $(-\Delta_p)^su=f$  for some $f\in L^1_{\rm loc}(\Omega)\cap (D^{s,p}_0(\Omega))^*$, in the sense that
\begin{equation}
\label{defweaksol}
\langle (-\Delta_p)^s u, \varphi\rangle=\int_{\Omega}f\,\varphi\, dx,\qquad \mbox{ for every } \varphi\in D^{s,p}_0(\Omega).
\end{equation}
 Let $v$ be a measurable function with compact support $K:=\mathrm{supp}(v)$ such that
\[
{\rm dist}(K,\Omega)>0,\qquad \int_{\Omega^c} |v|^{p-1}\, dx<+\infty,
\]
and define for a.e.\ Lebesgue point $x\in \Omega$ of $u$
\[
h(x)=2\int_{K}\frac{J_p\big(\big(u(x)-u(y)\big)-v(y)\big)-J_p\big(u(x)-u(y)\big)}{|x-y|^{N+s\,p}}\,dy.
\]
Then $u+v\in \widetilde{D}^{s,p}(\Omega)$ and $(-\Delta_p)^s(u+v)=f+h$ weakly.
\end{proposition}

\subsection{Radial functions}

For every measurable function $u:\mathbb{R}^N\to\mathbb{R}$ we define its distribution function
\[
\mu_u(t)=\big|\{x\, :\, |u(x)|>t\}\big|,\qquad t>0.
\] 
Let $0< q<\infty$ and $0<\theta<\infty$, the {\it Lorentz space} $L^{q,\theta}(\mathbb{R}^N)$ is defined by
\[
L^{q,\theta}(\mathbb{R}^N)=\left\{u\, :\, \int_0^\infty t^{\theta-1}\,\mu_u(t)^\frac{\theta}{q}\,dt<+\infty\right\}.
\]
In the limit case $\theta=\infty$, this is defined by
\[
L^{q,\infty}(\mathbb{R}^N)=\left\{u\, :\, \sup_{t>0} t\,\mu_u(t)^\frac{1}{q}<+\infty\right\},
\]
and we recall that this coincides with the weak $L^q$ space (see for example \cite[page 106]{liebL}).
\begin{lemma}[Radial Lemma for Lorentz spaces]
\label{lm:border}
Let $0<\theta\le \infty$ and $0< q<\infty$. Let $u\in L^{q,\theta}(\mathbb{R}^N)$ be a non-negative and radially symmetric decreasing function.
Then
\[
\begin{split}
0\le u(x)&\le \left(\theta\,\omega_N^{-\frac{\theta}{q}}\,\int_0^\infty t^{\theta-1}\,\mu_u(t)^\frac{\theta}{q}\,dt\right)^\frac{1}{\theta}\, |x|^{-\frac{N}{q}},\qquad \mbox{ if } \theta<\infty,\\
0\le u(x)&\le \left(\omega_N^{-\frac{1}{q}}\,\sup_{t>0} t\,\mu_u(t)^\frac{1}{q}\right)\,|x|^{-\frac{N}{q}},\qquad \mbox{ if } \theta=\infty.
\end{split}
\]
\end{lemma}
\begin{proof}
For simplicity, we simply write $\mu$ in place of $\mu_u$ and suppose that $u=u(r)$ coincides with its right-continuous representative. We start with the case $\theta<\infty$. 
First of all, we prove that
\begin{equation}
\label{equivalente}
\int_0^\infty t^{\theta-1}\,\mu(t)^\frac{\theta}{q}\,dt=\frac{N-\alpha}{N\,\theta\,\omega_N^{\alpha/N}}\,\int_{\mathbb{R}^N} \frac{u^\theta}{|x|^\alpha}\,dx,
\end{equation}
where the exponent $\alpha<N$ is given by the relation\footnote{Observe that if $\theta>q$, then $\alpha<0$.}
\[
\frac{\theta}{q}=\frac{N-\alpha}{N}.
\]
With a simple change of variable
\begin{equation}
\label{cambio}
\int_0^\infty t^{\theta-1}\,\mu(t)^\frac{\theta}{q}\, dt=\frac{1}{\theta}\, \int_0^\infty \mu(s^{1/\theta})^\frac{\theta}{q}\, ds.
\end{equation}
Then we observe that
\[
\int_{\mathbb{R}^N} \frac{u^\theta}{|x|^\alpha}\, dx=\int_{\mathbb{R}^N} \frac{\displaystyle\int_0^\infty \chi_{\{t\,:\,u(x)^\theta>t\}}(s)\, ds}{|x|^\alpha}\, dx=\int_0^\infty \int_{\mathbb{R}^N} \frac{\chi_{\{t\,:\,u(x)>t^{1/\theta}\}}(s)}{|x|^\alpha}\, dx\, ds,
\]
and
\[
\chi_{\{t\,:\,u(x)>t^{1/\theta}\}}(s)=\chi_{\{y\, :\, u(y)>s^{1/\theta}\}}(x).
\]
By assumption we have
\[
\{y\, :\, u(y)>s^{1/\theta}\}=\left\{y\, :\, |y|<\left(\frac{\mu(s^{1/\theta})}{\omega_N}\right)^\frac{1}{N}\right\}=:B_{R(s)},
\]
since the function $u$ is radially decreasing. Thus we arrive at
\[
\begin{split}
\int_{\mathbb{R}^N} \frac{u^\theta}{|x|^\alpha}\, dx&=\int_0^\infty \left(\int_{B_{R(s)}} \frac{1}{|x|^\alpha}\, dx\right)\, ds\\
&=N\,\omega_N\, \int_0^\infty \int_0^{R(s)} \varrho^{N-1-\alpha}\,d\varrho\, ds=\frac{N\,\omega_N}{N-\alpha}\, \int_0^\infty \frac{\mu(s^{1/\theta})^\frac{N-\alpha}{N}}{\omega_N^\frac{N-\alpha}{N}}\, ds\\
&=\frac{N\,\omega_N^{\alpha/N}}{N-\alpha}\, \int_0^\infty \mu(s^{1/\theta})^\frac{\theta}{q}\, ds.
\end{split}
\]
Using \eqref{cambio} we finally obtain
\[
\int_0^\infty t^{\theta-1}\,\mu(t)^\frac{\theta}{q}\, dt=\frac{N-\alpha}{N\,\theta\,\omega_N^{\alpha/N}}\, \int_{\mathbb{R}^N} \frac{u^\theta}{|x|^\alpha}\, dx,
\]
which proves \eqref{equivalente}.
\vskip.2cm\noindent
As for the decay estimate, thanks to \eqref{equivalente} we have
\[
\begin{split}
+\infty>\int_{\mathbb{R}^N} \frac{u^\theta}{|x|^\alpha}\,dx&=N\,\omega_N\, \int_0^{+\infty} u(\varrho)^\theta\,\varrho^{N-1-\alpha}\, d\varrho\\
&\ge N\,\omega_N\, \int_0^{R} u(\varrho)^\theta\, \varrho^{N-1-\alpha}\,d\varrho\ge N\,\omega_N\, u(R)^\theta\, \frac{R^{N-\alpha}}{N-\alpha},
\end{split}
\]
where $\alpha$ is as above. Recalling that $(N-\alpha)/\theta=N/q$, we get the desired conclusion.
\vskip.2cm\noindent
For the case $\theta=\infty$, it is sufficient to observe that
\[
\sup_{t>0} t\,\mu(t)^\frac{1}{q}=\omega_N^\frac{1}{q}\,\sup_{x\in\mathbb{R}^N} |x|^\frac{N}{q}\,u(x).
\]
Then the decay estimate easily follows.
\end{proof}

\section{Properties of extremals}

\label{sec:3}

\subsection{Basic properties}
We first observe that by homogeneity we can equivalently write
\begin{equation}
\label{costante}
\mathcal{S}_{p,s}=\inf_{u\in D^{s,p}(\mathbb{R}^N)} \left\{\int_{\mathbb{R}^{2N}}\frac{|u(x)-u(y)|^p}{|x-y|^{N+s\,p}}\,dx\,dy\, :\, \int_{\mathbb{R}^N} |u|^{{N\,p/(N-s\,p)}}\,dx=1\right\}.
\end{equation}
Then we start with the following result.
\begin{proposition}
\label{prop:4cazzate}
Let $1<p<\infty$ and $s\in(0,1)$ be such that $s\,p<N$. Then:
\begin{itemize}
\item problem \eqref{costante} admits a solution;	
\vskip.2cm
\item for every $U\in D^{s,p}(\mathbb{R}^N)$ solving \eqref{costante}, there exist $x_0\in\mathbb{R}^N$ and $u:\mathbb{R}^+\to\mathbb{R}$ constant sign monotone function such that $U(x)=u(|x-x_0|)$; 
\vskip.2cm
\item every minimizer $U\in D^{s,p}(\mathbb{R}^N)$ weakly solves
\[
(-\Delta_p)^s U=\mathcal{S}_{p,s}\,|U|^{p^*-2}\,U,\qquad \mbox{ in }\mathbb{R}^N,		
\]
that is
\begin{equation}
\label{eulero}
\int_{\mathbb{R}^{2N}}  \frac{J_p(U(x)-U(y))\, \big(\varphi(x)-\varphi(y)\big)}{|x-y|^{N+s\,p}}\,dx\,dy=\mathcal{S}_{p,s}\, \int_{\mathbb{R}^N} |U|^{p^*-2}\,U\,\varphi\,dx,
\end{equation}
for every $\varphi\in D^{s,p}(\mathbb{R}^N)$.
\end{itemize}
\end{proposition}
\begin{proof}
The existence of a solution for \eqref{costante} follows from the Concentration-Compactness Principle, see \cite[Section I.4, Example iii)]{Lions}.
\par
It is not difficult to show that every solution of \eqref{costante} must have costant sign. Indeed, for every admissible $u\in D^{s,p}(\mathbb{R}^N)$, still by \eqref{valorassoluto} the function $|u|$ is still admissible and does not increase the value of the functional.
More important, the inequality sign in \eqref{valorassoluto} is strict if $u(x)\,u(y)<0$, i.e. if $u$ changes sign. 
\par
Radial symmetry of the solutions comes from the {\it P\'olya-Szeg\H{o} principle} for Gagliardo seminorms (see \cite{AlLi}), i.e. for every non-negative function $u\in D^{s,p}(\mathbb{R}^N)$ we have
\begin{equation}
\label{sp}
[u^\#]^p_{s,p}\le [u]^p_{s,p}.
\end{equation}
Here $u^\#$ denotes the radially symmetric decreasing rearrangement of $u$. It is crucial to observe that inequality \eqref{sp} is strict, unless $u$ is (up to a translation) a radially symmetric decreasing function, see \cite[Theorem A.1]{FS}.
\par
Finally, if $U$ solves \eqref{costante}, then it minimizes as well the functional
\[
u\mapsto [u]^p_{s,p}-\mathcal{S}_{p,s}\,\left(\int_{\mathbb{R}^N} |u|^{p^*}\,dx\right)^\frac{p}{p^*}.
\]
Equation \eqref{eulero} is exactly the Euler-Lagrange equation associated with this functional, once it is observed that $U$ has unitary $L^{p^*}$ norm.
\end{proof}
\begin{proposition}[Global boundedness]
Let $U\in D^{s,p}(\mathbb{R}^N)$ be a non-negative solution of \eqref{costante}. Then we have $U\in L^\infty(\mathbb{R}^N)\cap C^0(\mathbb{R}^N)$.
\end{proposition}
\begin{proof}
Thanks to the properties of the minimizers contained in Proposition \ref{prop:4cazzate}, it is  enough to prove that $U\in L^\infty_{\mathrm{\mathrm{loc}}}(\mathbb{R}^N)$, since continuity then follows from \cite[Theorem 3.13]{BP} (see also \cite[Theorem 5.4]{IMS} for a direct proof).
With this aim, we just need to show that $U\in L^{q\,(p^*-1)}(\mathbb{R}^N)$ for some $q>N/(s\,p)$. This would imply that
\[
U^{p^*-1}\in L^q(\mathbb{R}^N),\qquad \mbox{ for some }q>\frac{N}{s\,p}, 
\]
and thus $U\in L^\infty_{\mathrm{loc}}(\mathbb{R}^N)$ would automatically follow by \cite[Theorem 3.8]{BP}.
\vskip.2cm\noindent
Let $M>0$ and $\alpha>1$, we set for simplicity $U_M=\min\{U,M\}$ and $g_{\alpha,M}(t)=t\,\min\{t,\,M\}^{\alpha-1}$. Then we insert in \eqref{eulero} the test function $\varphi=g_{\alpha,M}(U)\in D^{s,p}(\mathbb{R}^N)$. This yields
\[
\begin{split}
\int_{\mathbb{R}^{2N}} & \frac{J_p(U(x)-U(y))\, \big(g_{\alpha,M}(U(x))-g_{\alpha,M}(U(y))\big)}{|x-y|^{N+s\,p}}\,dx\,dy=\mathcal{S}_{p,s}\, \int_{\mathbb{R}^N} U^{p^*-p}\,U_{M}^{\alpha-1}\,U^p\,dx.\\
\end{split}
\]
We now observe that if we set
\[
G_{\alpha,M}(t)=\int_0^t g'_{\alpha,M}(\tau)^\frac{1}{p}\,d\tau,
\]
by using \cite[Lemma A.2]{BP} from the previous identity with simple manipulations we get
\[
\begin{split}
\int_{\mathbb{R}^{2N}} &\frac{|G_{\alpha,M}(U(x))-G_{\alpha,M}(U(y))|^p}{|x-y|^{N+s\,p}}\,dx\,dy\\
&\le \mathcal{S}_{p,s}\, \left[K_0^{\alpha-1}\,\int_{\mathbb{R}^N} U^{p^*}\,dx+\left(\int_{\{U\ge K_0\}} U^{p^*}\,dx\right)^\frac{p^*-p}{p^*}\,\left(\int_{\mathbb{R}^N} \left(U_{M}^{(\alpha-1)}\,U^{p}\right)^\frac{p^*}{p}\,dx\right)^\frac{p}{p^*}\right],
\end{split}
\]
for some $K_0>0$ that will be chosen in a while.
If we estimate from below the left-hand side by Sobolev inequality and use that $U$ has unitary norm, we get\footnote{Here we use that 
\[
G_{\alpha,M}(t)\ge \frac{p}{p+\alpha-1}\, t\,\min\{t,\,M\}^\frac{\alpha-1}{p}.	
\]}
\begin{equation}
\label{intermediate}
\begin{split}
\left(\frac{p}{p+\alpha-1}\right)^p\,\left(\int_{\mathbb{R}^N} \right.&\left.\left(U^p\,U_M^{(\alpha-1)}\right)^\frac{p^*}{p}\,dx\right)^\frac{p}{p^*}\le K_0^{\alpha-1}\, \\
&+\left(\int_{\{U\ge K_0\}} U^{p^*}\,dx\right)^\frac{p^*-p}{p^*}\,\left(\int_{\mathbb{R}^N} \left(U_M^{(\alpha-1)}\,U^{p}\right)^\frac{p^*}{p}\,dx\right)^\frac{p}{p^*}.
\end{split}
\end{equation}
We now choose the parameters: we first take $\alpha>1$ such that
\[
p^*+(\alpha-1)\,\frac{p^*}{p}=q\,(p^*-1),\qquad \mbox{ i.\,e. } \quad \alpha=p\,q\,\frac{(p^*-1)}{p^*}-(p-1),
\]
where $q>N/(s\,p)$,
then we choose $K_0=K_0(\alpha,U)>0$ such that
\[
\left(\int_{\{U\ge K_0\}} U^{p^*}\,dx\right)^\frac{p^*-p}{p^*}\le \frac{1}{2}\,\left(\frac{p}{p+\alpha-1}\right)^p.
\]
With this choice we can absorb the last term on the right-hand side of \eqref{intermediate} and thus obtain
\[
\left(\frac{p}{p+\alpha-1}\right)^p\,	\left(\int_{\mathbb{R}^N} U^{p^*}\,U_M^{(\alpha-1)\,\frac{p^*}{p}}\,dx\right)^\frac{p}{p^*}\le 2\,K_0^{\alpha-1}.
\]
If we now take the limit as $M$ goes to $+\infty$, we finally get that $U\in L^{q\,(p^*-1)}(\mathbb{R}^N)$ for some $q>N/(s\,p)$, together with the estimate
\[
\left\|U^{p^*-1}\right\|^q_{q}\le \left(2\,K_0^{\alpha-1}\,\left(\frac{p+\alpha-1}{p}\right)^p\right)^\frac{p^*}{p},
\] 
and thus the conclusion.
\end{proof}

\begin{proposition}[Borderline Lorentz estimate]
\label{lp-weak}
Let $U\in D^{s,p}(\mathbb{R}^N)$ be a non-negative solution of \eqref{costante}. Then
\begin{equation}
\label{cazzata}
U\in L^q(\mathbb{R}^N),\qquad \mbox{ for every } q>q_0:=\frac{(p-1)\,N}{N-s\,p}.
\end{equation}
Moreover, we have $U\in L^{q_0,\infty}(\mathbb{R}^N)$ with the estimate
\begin{equation}
\label{stimalpdebole}
\sup_{t>0}t\,|\{U> t\}|^\frac{1}{q_0}\leq \|U\|_{p^*-1}^{\frac{p^*-1}{p-1}}.
\end{equation}
\end{proposition}
\begin{proof}
We divide the proof in two parts: we first prove \eqref{cazzata}.
Then we will use \eqref{cazzata} to prove \eqref{stimalpdebole}.
\vskip.2cm\noindent
{\it Part I: intermediate estimate}.
Given $0<\alpha<1$ and 
$\varepsilon>0$, we take the Lipschitz increasing function $\psi_\varepsilon:[0,+\infty)\to[0,+\infty)$ 
defined as
\[
\psi_\varepsilon(t)=\int_0^t \left[(\varepsilon+\tau)^\frac{\alpha-1}{p}+\frac{\alpha-1}{p}\,\tau\,(\varepsilon+\tau)^\frac{\alpha-1-p}{p}\right]^p\,d\,\tau.
\]
We observe that
\begin{equation}
\label{geps}
0\le \psi_\varepsilon(t)\le \int_0^t (\varepsilon+t)^{\alpha-1}\,d\tau=\frac{1}{\alpha}\,[(\varepsilon+t)^\alpha-\varepsilon^\alpha]\le \frac{t^\alpha}{\alpha},
\end{equation}
where in the second inequality we used that $0<\alpha<1$.
We insert in \eqref{eulero} the test function $\varphi=\psi_\varepsilon(U)\in D^{s,p}(\mathbb{R}^N)$. This gives
\[
\begin{split}
\int_{\mathbb{R}^{2N}} & \frac{J_p(U(x)-U(y))\, \big(\psi_{\varepsilon}(U(x))-\psi_\varepsilon(U(y)\big)}{|x-y|^{N+s\,p}}\,dx\,dy=\mathcal{S}_{p,s}\, \int_{\mathbb{R}^N} U^{p^*-1}\,\psi_\varepsilon(U)\,dx.\\
\end{split}
\]
By defining
\[
\Psi_\varepsilon(t):=\int_0^t \psi'_\varepsilon(\tau)^\frac{1}{p}\,d\tau	=t\,(\varepsilon+t)^\frac{\alpha-1}{p},
\]
if we proceed as in the previous proof and use \eqref{geps}, we get
\[
\begin{split}
\left(\int_{\mathbb{R}^N} \Psi_\varepsilon(U)^{p^*}\,dx\right)^\frac{p}{p^*}&\le \frac{1}{\alpha}\,\|U\|^{p^*+\alpha-1}_{\infty}\,|\{U>K_0\}|\\
&\quad +\left(\int_{\{U\le K_0\}} U^{p^*}\,dx\right)^\frac{p^*-p}{p^*}\left(\int_{\{U\le K_0\}} \Big(\psi_\varepsilon(U)\,U^{p-1}\Big)^\frac{p^*}{p}\,dx\right)^\frac{p}{p^*},
\end{split}
\] 
for $K_0>0$. Observe that we also used the previous Proposition to assure that $U\in L^\infty(\mathbb{R}^N)$. 
From \eqref{geps} we get
\begin{equation*}
0\le \psi_\varepsilon(t)\,t^{p-1}\le \frac{1}{\alpha}\,[(\varepsilon+t)^\alpha-\varepsilon^\alpha]\,t^{p-1}
\le \quad \frac{1}{\alpha}\,(\varepsilon+t)^{\alpha-1}\,t^{p}=\frac{1}{\alpha}\, \Psi_\varepsilon(t)^p.
\end{equation*}
Thus we arrive at 
\[
\begin{split}
\left(\int_{\mathbb{R}^N} \Psi_\varepsilon(U)^{p^*}\,dx\right)^\frac{p}{p^*}&\le \frac{1}{\alpha}\,\|U\|^{p^*+\alpha-1}_{\infty}\,|\{U>K_0\}|\\
&+\frac{1}{\alpha}\,\left(\int_{\{U\le K_0\}} U^{p^*}\,dx\right)^\frac{p^*-p}{p^*}\,\left(\int_{\mathbb{R}^N} \Psi_\varepsilon(U)^{p^*}\,dx\right)^\frac{p}{p^*},
\end{split}
\]
The level $K_0=K_0(\alpha,U)>0$ is now chosen so that
\[
\left(\int_{\{U\le K_0\}} U^{p^*}\,dx\right)^\frac{p^*-p}{p^*}\le \frac{\alpha}{2},
\]
which yields
\[
\left(\int_{\mathbb{R}^N} \left(U\,(U+\varepsilon)^\frac{\alpha-1}{p}\right)^{p^*}\,dx\right)^\frac{p}{p^*}\le \frac{2}{\alpha}\,\|U\|^{p^*+\alpha-1}_{\infty}\,|\{U>K_0\}|,
\]
for every $0<\alpha<1$. By taking the limit as $\varepsilon$ goes to $0$, we get the desired integrability \eqref{cazzata}. 
\vskip.2cm\noindent
{\it Part II: borderline Lorentz estimate.} We now prove \eqref{stimalpdebole}. 
For any $t>0$ we let $g_t(s)=\min\{t,\,s\}$, and define 
\[
G_t(s)=\int_0^sg'_t(\tau)^{\frac 1 p}d\tau=g_t(s).
\]
We test \eqref{eulero} with $g_t(U)$ and, thanks to \cite[Lemma A.2]{BP} and Sobolev inequality we get
\[
\begin{split}
\mathcal{S}_{p,s}\,\|g_t(U)\|_{p^*}^p&\leq [g_t(U)]_{s, p}^p\leq \int_{\mathbb{R}^{2N}} \frac{J_p(U(x)-U(y))\, \big(g_t(U(x))-g_t(U(y))\big)}{|x-y|^{N+s\,p}}\,dx\,dy\\
&\leq \mathcal{S}_{p,s}\,\int_{\mathbb{R}^N} U^{p^*-1}\,g_t(U)\,dx.
\end{split}
\]
We have $U\in L^{p^*-1}(\mathbb{R}^N)$, by choosing  $q=p^*-1$ in \eqref{cazzata}. Thus we conclude that
\[
t\,|\{U>t\}|^{\frac{1}{p^*}}\leq \|g_t(U)\|_{p^*}\leq \left(\int_{\mathbb{R}^N} U^{p^*-1}g_t(U)\,dx\right)^\frac{1}{p}\leq t^\frac{1}{p}\,\|U\|_{p^*-1}^\frac{p^*-1}{p}.
\]
This finally yields \eqref{stimalpdebole}, after some elementary manipulations.
\end{proof}

\subsection{Decay estimates}

As an intermediate step towards the proof of the asymptotic result \eqref{sciarpa2}, in this subsection we will prove that any (positive) solution of \eqref{costante} verifies
\[
\frac{1}{C}\,|x|^{-\frac{N-s\,p}{p-1}}\le U(x)\le C\,|x|^{-\frac{N-s\,p}{p-1}},\qquad |x|> 1,
\]
for some $C=C(N,p,s,U)>1$, 
see Corollary \ref{lower-Dec} below.
\vskip.2cm\noindent
In what follows, we will set for simplicity
\[
\Gamma(x)=|x|^{-\frac{N-s\,p}{p-1}},\qquad x\in\mathbb{R}^N\setminus\{0\},
\]
and
\begin{equation}
\label{stroncata}
\widetilde \Gamma(x)=\min\{1,\,\Gamma(x)\}=\min\left\{1,\, |x|^{-\frac{N-s\,p}{p-1}}\right\},\qquad x\in\mathbb{R}^N.
\end{equation}
The following expedient result will be useful.
\begin{lemma}
\label{lm:stroncalaplaciani!}
With the notation above, we have
\begin{equation}
\label{splv}
\frac{1}{C}\,\,|x|^{-N-s\,p}\leq (-\Delta_p)^s \widetilde\Gamma(x)\leq C\,|x|^{-N-s\,p},\qquad  \mbox{ for }|x|>R>1,
\end{equation}
in weak sense, for some $C=C(N,p,s,R)>1$. The constant blows-up as $R$ goes to $1$.
\end{lemma}
\begin{proof}
From Theorem \ref{fundamental}, we know that $\Gamma$ belongs to $\widetilde{D}^{s,p}(B_R^c)$ and
is a weak solution of $(-\Delta_p)^su=0$ in $B_R^c$ for any $R>1$. We then observe that the truncated function $\widetilde \Gamma$ can be written as
\[
\widetilde \Gamma(x)=\Gamma(x)-\big(\Gamma(x)-1\big)_+.
\]
Thus we apply Proposition \ref{nonlocalb}, with the choices
\[
\Omega=B_R^c,\qquad u=\Gamma,\qquad f\equiv 0,\qquad v=-(\Gamma-1)_+,
\] 
This yields for $|x|>R$
\begin{equation}
\label{stroncalaplaciano}
\begin{split}
(-\Delta_p)^s \widetilde\Gamma (x)&=2\,\int_{B_1}\frac{J_p(\Gamma(x)-1)-J_p(\Gamma(x)-\Gamma(y))}{|x-y|^{N+s\,p}}\, dy\\
&=2\,\int_{B_1}\frac{J_p(\Gamma(y)-\Gamma(x))-J_p(1-\Gamma(x))}{|x-y|^{N+s\,p}}\, dy.
\end{split}
\end{equation}
We first prove the upper bound in \eqref{splv}. To this aim, by the monotonicity of $\Gamma$ we get
\[
(\Gamma(y)-\Gamma(x))^{p-1}-(1-\Gamma(x))^{p-1}\leq (\Gamma(y)-\Gamma(x))^{p-1}\leq \Gamma(y)^{p-1},\qquad |x|> R, \ |y|\leq 1.
\]
Moreover 
\[
|x-y|\geq \frac{R-1}{R}\,|x|,\qquad \mbox{ for all } |x|>R \mbox{ and } |y|<1.
\]
By spending these informations in \eqref{stroncalaplaciano}, we obtain
\[
(-\Delta_p)^s \widetilde\Gamma (x)\leq \left(\frac{R}{R-1}\right)^{N+s\,p}\,\frac{2}{|x|^{N+s\,p}}\int_{B_1} \Gamma(y)^{p-1}\, dy=\frac{C}{|x|^{N+s\,p}},
\]
as desired.
Observe that we also used that $\Gamma\in L^{p-1}_{\mathrm{loc}}(\mathbb{R}^N)$.
\vskip.2cm\noindent
In order to prove the lower bound, we need to distinguish between the case $1<p<2$ and the case $p\ge 2$. If $p\ge 2$, then $J_p(t)=|t|^{p-2}\,t$ is a convex superadditive function on $[0,\infty)$. Thus we get
\[
J_p(\Gamma(y)-\Gamma(x))-J_p(1-\Gamma(x))\ge J_p(\Gamma(y)-1),\qquad |x|>R>1>|y|.
\]
As for the kernel, we have
\begin{equation}
\label{boffo}
|x-y|<2\,|x|,\qquad \mbox{ if }|x|>|y|,
\end{equation}
thus in conclusion from \eqref{stroncalaplaciano} we get
\[
(-\Delta_p)^s \widetilde\Gamma (x)\ge \frac{2^{1-N-s\,p}}{|x|^{N+s\,p}}\int_{B_1} \big(\Gamma(y)-1\big)^{p-1}\, dy=\frac{C}{|x|^{N+s\,p}}.
\]
Using again that $\Gamma\in L^{p-1}_{\mathrm{loc}}(\mathbb{R}^N)$ and that $\Gamma>1$ in $B_1$ gives the lower bound in \eqref{splv}, in the case $p\ge 2$.
\par
In the case $1<p<2$, we need to use \eqref{seconda}, which gives
\[
\begin{split}
J_p(\Gamma(y)-\Gamma(x))-J_p(1-\Gamma(x))&\ge C\,\frac{(\Gamma(y)-1)}{\Big((\Gamma(y)-\Gamma(x))^2+(1-\Gamma(x))^2\Big)^\frac{2-p}{2}}\\
&\ge \frac{C}{2^\frac{2-p}{2}}\,\frac{(\Gamma(y)-1)}{(\Gamma(y)-\Gamma(x))^{2-p}}\ge C\,\Gamma(y)^{p-1}\, \left(1-\frac{1}{\Gamma(y)}\right).
\end{split}
\]
By using this and \eqref{boffo} in \eqref{stroncalaplaciano}, we get the desired lower bound for $1<p<2$ as well.
\end{proof}
\noindent
In order to prove a lower bound for positive radially decreasing solutions of \eqref{costante}, we need to focus on the auxiliary problem
\begin{equation}
\label{cappot}
\mathcal{I}(R)=\inf_{u\in D^{s,p}(\mathbb{R}^N)}\Big\{[u]_{s, p}^p: u\geq \chi_{B_R}\Big\}.
\end{equation}
\begin{proposition}
Let $1<p<\infty$ and $s\in(0,1)$ be such that $s\,p<N$. For any $R>0$, problem \eqref{cappot} has a unique solution $u_R>0$. Moreover, $u_R$ is radial, non-increasing and  $u_R\in D^{s,p}(\mathbb{R}^N)$ solves in weak sense
\[
\left\{\begin{array}{rcll}
(-\Delta_p)^s u_R&=&0&\text{ in }  \overline{B_R}^{\,c},\\
u_R&\equiv& 1&\mbox{ in } \overline{B_R}.
\end{array}
\right.
\]
\end{proposition}
\begin{proof}
The existence of a solution follows easily by using the Direct Methods. Indeed, if $\{u_n\}_{n\in\mathbb{N}}\subset D^{s,p}(\mathbb{R}^N)$ is a minimizing sequence, then a uniform bound on their Gagliardo seminorms entails a uniform bound on the $L^{p^*}$ norms, by Sobolev inequality. Thus we have weak convergence (up to a subsequence) in $L^{p^*}(\mathbb{R}^N)$ to a function $u\in D^{s,p}(\mathbb{R}^N)$. Moreover, the constraint $u_n\ge \chi_{B_R}$ is stable with respect to weak convergence and thus it passes to the limit. Consequently, $u$ is a minimizer. The uniqueness follows from strict convexity of the functional.
\par
All the other required properties of $u_R$ follow as in the proof of Proposition \ref{prop:4cazzate}, we just show that $u_R$ saturates the constraint $u_R\ge \chi_{B_R}$. For simplicity, we set $\mathcal{E}(u)=[u]_{s,p}^p$. Then from \cite[Remark 3.3]{GM} we have
\begin{equation}
\label{submodular}
\mathcal{E}(\max\{u,\, t\})+\mathcal{E}(\min\{u,\, t\})\leq \mathcal{E}(u),\qquad \mbox{ for every } u\in D^{s,p}(\mathbb{R}^N),\ t\in \mathbb{R}.
\end{equation}
In particular, $\min\{u_R,1\}$ is admissible and is still a minimizer. Thus by uniqueness it coincides with $u_R$.
\end{proof}
Thanks to Lemma \ref{lm:stroncalaplaciani!}, we can prove a decay estimate for the solution of \eqref{cappot}. 
\begin{proposition}
\label{prop:potenziale}
The solution $u_1$ of problem \eqref{cappot} with $R=1$ satisfies
\[
\frac{|x|^{-\frac{N-s\,p}{p-1}}}{C}\leq u_1(x)\leq p^\frac{1}{p-1}\,|x|^{-\frac{N-s\,p}{p-1}},\qquad \mbox{ for } |x|\geq 1,
\]
for some constant $C=C(N,p,s)>1$.
\end{proposition}

\begin{proof} Observe that $u_1$ is continuous due to \cite[Theorem 1.1]{IMS}. We prove the two estimates separately.
\vskip.2cm\noindent
{\em Upper bound.} 
We first observe that by using the scaling properties of the Gagliardo seminorm, we have
\begin{equation}
\label{scaling}
\mathcal{I}(R)=R^{N-s\,p}\,\mathcal{I}(1).
\end{equation}For every $R>1$, we set $u_1(R)=t \in (0,1)$. As in the previous proof, we set $\mathcal{E}(u)=[u]_{s,p}^p$.
The function $\min\{u_1,\, t\}/t$ is admissible for problem \eqref{cappot} with $B_R$, then the minimality of $u_R$ gives
\[
\mathcal{E}\left(\frac{\min\{u_1,\,t\}}{t}\right)\geq \mathcal{E}(u_R)=\mathcal{I}(R)=R^{N-s\,p}\,\mathcal{I}(1),
\]
thanks to \eqref{scaling}.
Similarly, we get
\[
\mathcal{E}\left(\frac{\max\{u_1-t,\, 0\}}{1-t}\right)\geq \mathcal{E}(u_1)=\mathcal{I}(1).
\]
then using the $p-$homogeneity of the energy and summing the previous two inequalities 
\[
\mathcal{E}(\min\{u_1,\, t\})+\mathcal{E}(\max\{u_1,\, t\})\geq \big(t^p\,R^{N-s\,p}+(1-t)^p\big)\,\mathcal{I}(1).
\]
Using the submodularity of Gagliardo seminorms \eqref{submodular} in the left-hand side and simplifying we get
\[
t^p\,R^{N-s\,p}\leq 1-(1-t)^p.
\]
By recalling the definition of $t$, we obtain
\begin{equation}
\label{decu1}
u_1(R)^p\, R^{N-s\, p}\leq 1-(1-u_1(R))^p
\end{equation}
and since $1-(1-u_1(R))^p\leq p\, u_1(R)$ we get
\[
u_1(R)\leq p^{\frac{1}{p-1}}\, R^{-\frac{N-sp}{p-1}}.
\]
\vskip.2cm\noindent
{\em Lower bound}. 
By using Proposition \ref{nonlocalb} with
\[
\Omega=\overline{B_3}^{\,c},\qquad u=u_1,\qquad f\equiv 0,\qquad v=-(u_1-u_1(2))_+,
\] 
the truncated function 
\[
u=\min\{u_1,\,u_1(2)\}=u_1-(u_1-u_1(2))_+,
\] 
satisfies weakly in $\overline{B_3}^{\,c}$
\[
\begin{split}
(-\Delta_p)^s u(x)&=2\int_{B_2}\frac{J_p(u_1(x)-u_1(2))-J_p(u_1(x)-u_1(y))}{|x-y|^{N+s\,p}}\, dy\\
&\geq 2\int_{B_1}\frac{J_p(u_1(y)-u_1(x))-J_p(u_1(2)-u_1(x))}{|x-y|^{N+s\,p}}\, dy.
\end{split}
\]
In the last passage we used that the integrand is nonnegative by the monotonicity of $u_1$. Recall that $u_1\equiv 1$ in $B_1$ and by \eqref{decu1} we have $u_1(2)<1=u_1(1)$. Then, it is readily checked that  
\[
(u_1(1)-u_1(x))^{p-1}-(u_1(2)-u_1(x))^{p-1}\geq c,\qquad  \mbox{ for }|x|>3,
\]
for some constant $c=c(p,u_1(1)-u_1(2))>0$.
Since also $|x-y|\leq 2\,|x|$ for all $x\in B_2^c$ and $y\in B_1$, the previous discussion yields
\begin{equation}
\label{stimau}
(-\Delta_p)^s u(x)\geq \frac{2\,c\,|B_1|}{(2\,|x|)^{N+s\,p}}=:\frac{c_1}{|x|^{N+s\,p}},\qquad \mbox{ for }|x|>3.
\end{equation} 
On the other hand, from Lemma \ref{lm:stroncalaplaciani!}, for every $\varepsilon>0$ we have \begin{equation}
\label{stimavepsilon}
(-\Delta_p)^s (\varepsilon\,\widetilde\Gamma (x))\le \frac{c_2}{|x|^{N+s\,p}}\,\varepsilon^{p-1},\qquad \mbox{ for } |x|>3,
\end{equation}
where $\widetilde \Gamma$ is given in \eqref{stroncata}.
Now choose $\varepsilon>0$ as follows
\[
\varepsilon=\min\left\{u_1(3)\,3^\frac{N-s\,p}{p-1},\, \left(\frac{c_1}{c_2}\right)^\frac{1}{p-1}\right\},
\]
so that by \eqref{stimau} and \eqref{stimavepsilon} it holds
\[
\left\{\begin{array}{rcll}
(-\Delta_p)^s (\varepsilon\,\widetilde\Gamma)\!\!\! &\leq&\!\!\! (-\Delta_p)^s u&\text{ in } \overline{B_3}^{\,c},\\
\varepsilon\,\widetilde\Gamma\!\!\!&\leq&\!\!\! u &\text{ in } \overline{B_3}.
\end{array}
\right.
\]
Therefore by Theorem \ref{comparison} and the definitions of $\widetilde \Gamma$ and $u$ we have 
\[
\varepsilon\,|x|^{-\frac{N-s\,p}{p-1}}\leq u(x)=u_1(x),\quad \mbox{ for }|x|>3.
\]
In $\overline{B_3}\setminus B_1$ the estimate is simpler to obtain, indeed
\[
u_1(3)\,|x|^{-\frac{N-s\,p}{p-1}}\le u_1(3)\le u_1(x),
\]
thus we get the conclusion.
\end{proof}
Finally, we can prove the aforementioned decay estimate for solutions of \eqref{costante}.
\begin{corollary}[Sharp decay rate]
\label{lower-Dec}
Let $U\in D^{s,p}(\mathbb{R}^N)$ be a positive radially symmetric 
and decreasing solution of \eqref{costante}. Then
\[
\left(\inf_{B_1} U\right)\,\frac{|x|^{-\frac{N-s\,p}{p-1}}}{C}\le U(x)\le \left(\omega_N^{-\frac{1}{p^*}}\,\|U\|_{p^*-1}^{\frac{p^*-1}{p}}\right)^\frac{p}{p-1}\,|x|^{-\frac{N-s\,p}{p-1}},\qquad |x|\geq 1,
\]
where the constant $C=C(N,p,s)>1$ is the same of Proposition \ref{prop:potenziale}.
\end{corollary}

\begin{proof}
The upper bound follows from the borderline $L^{q_0,\infty}$ estimate of \eqref{stimalpdebole}, combined with the Radial Lemma \ref{lm:border}.
\par
As for the lower bound, by the weak Harnack inequality for positive supersolution of $(-\Delta_p)^s$ (see \cite[Theorem 5.2]{IMS}), we have 
\[
\lambda:=\inf_{B_1} U\ge C\,\left(\int_{B_2} U^{p-1}\,dx\right)^\frac{1}{p-1}>0.
\]
Then the function $\lambda\, u_1$ is a lower barrier for $U$ in $B_1^c$. Thus the lower bound follows from Theorem \ref{comparison} and Proposition \ref{prop:potenziale}.
\end{proof}

\section{Proof of the main result}

\label{sec:4}
In this section we still denote by $\widetilde \Gamma$ the truncated function defined by \eqref{stroncata}, while $U$ is a positive radially symmetric 
and decreasing solution of \eqref{costante}. As in the previous section we will systematically use the abuse of notation $U(x)=U(r)$ and $\widetilde \Gamma(x)=\widetilde \Gamma(r)$, for $r=|x|$.

\begin{lemma}
\label{lm:intervalli}
Suppose that
$$
U(R)\geq A\,\widetilde \Gamma(R)\qquad\text{for some $R>2$}. 
$$
For any $\delta>0$ there exists $\theta=\theta(N, p, s, \delta,U)<1$ such that 
\[
U(r)\geq (A-\delta)\,\widetilde\Gamma(r)\qquad \mbox{ for any }\theta\, R\leq r\leq R.
\]
Similarly, if 
$$
U(R)\leq B\, \widetilde \Gamma(R)\qquad \text{for some $R>2$}, 
$$
then 
\[
U(r)\leq (B+\delta)\, \widetilde\Gamma(r)\qquad \mbox{ for any }R\leq r\leq R/\theta.
\]
\end{lemma}

\begin{proof}
Consider the first statement and let $\theta<1$ to be determined. $U$ is non increasing and 
$$
U(r)\leq C\,r^{-\frac{N-s\,p}{p-1}}, \qquad r\ge 1,
$$ 
by Corollary \ref{lower-Dec}.
Then for any $\theta\, R\leq r\leq R$ it holds
\[
\frac{U(R)}{\widetilde\Gamma(R)}-\frac{U(r)}{\widetilde\Gamma(r)}\leq U(R)\,\left(\frac{1}{\widetilde\Gamma(R)}-\frac{1}{\widetilde\Gamma(r)}\right)\leq \frac{C}{R^{\frac{N-s\,p}{p-1}}}\,\left(R^{\frac{N-s\,p}{p-1}}-r^{\frac{N-s\,p}{p-1}}\right)\leq C\,\left(1-\theta^{\frac{N-s\,p}{p-1}}\right).
\]
Therefore by hypothesis we get
\[
\frac{U(r)}{\widetilde\Gamma(r)}\geq A-C\,\left(1-\theta^{\frac{N-s\,p}{p-1}}\right),\qquad \mbox{ for }  r\in [\theta\, R, R],
\]
which gives the first claim. The proof of the other statement is similar: for any $R\leq r\leq R/\theta$ it holds
\[
\frac{U(r)}{\widetilde\Gamma(r)}-\frac{U(R)}{\widetilde\Gamma(R)}\leq U(R)\,\left(\frac{1}{\widetilde\Gamma(r)}-\frac{1}{\widetilde\Gamma(R)}\right)\leq  \frac{C}{R^{\frac{N-sp}{p-1}}}\,\left(r^{\frac{N-s\,p}{p-1}}-R^{\frac{N-s\,p}{p-1}}\right)\leq C\,\left(\theta^{-\frac{N-s\,p}{p-1}}-1\right),
\]
which gives
\[
\frac{U(r)}{\widetilde\Gamma(r)}\leq B+C\,\left(\theta^{-\frac{N-sp}{p-1}}-1\right),\qquad \mbox{ for } r\in [R, R/\theta].
\]
This completes the proof.
\end{proof}

\noindent
We are ready for the proof of the main result. 
\begin{theorem}
There exists $U_\infty>0$ such that
$$
\lim_{r\to +\infty} r^{\frac{N-s\,p}{p-1}}U(r)=U_\infty.
$$ 
\end{theorem}
\begin{proof}
We can suppose that $p\neq 2$, since for $p=2$ the function $U$ has an explicit expression.
By virtue of Corollary \ref{lower-Dec} we readily have
\[
\frac{1}{C}\leq m:=\liminf_{r\to +\infty} \frac{U(r)}{\widetilde\Gamma(r)}\leq \limsup_{r\to +\infty}\frac{U(r)}{\widetilde\Gamma(r)}=:M\leq C,
\]
with $C$ depending on $U$ as well.
Suppose by contradiction that $M-m>0$, and fix $0<\varepsilon_0<(M-m)/4$.  
\vskip.2cm\noindent
$\bullet$ \underline{\em Case $p>2$}.
There exists $R_0=R_0(\varepsilon_0)>2$ such that 
\[
\frac{U(r)}{\widetilde\Gamma(r)}\geq m-\varepsilon_0,\quad \mbox{ for }r\geq R_0,
\]
and we can choose an arbitrarily large $R>R_0$ such that 
$$
\frac{U(R)}{\widetilde\Gamma(R)}\geq M-\frac{M-m}{4}.
$$ 
Consider $\delta=(M-m)/4$. By Lemma \ref{lm:intervalli}, there exists $\theta<1$ so that for any such $R$ it holds
\begin{equation}
\label{uvgeq}
\frac{U(r)}{\widetilde\Gamma(r)}\geq \frac{M+m}{2},\qquad \mbox{ for }  r\in [\theta\, R, R].
\end{equation}
Since $R$ can be chosen arbitrarily large, we can suppose 
$\theta R>R_0$ as well. Consider, for any $0<\varepsilon<(M-m)/4$, the lower barrier $w(r)=g(r)\,\widetilde\Gamma(r)$ where $g$ is the following step function
\[
g(r)=\begin{cases}
0&\text{if $r<R_0$},\\
m-\varepsilon_0&\text{if $R_0\leq r<\theta R$},\\
\frac{M+m}{2}&\text{if $\theta R\leq r<\sqrt{\theta}R$},\\
m+\varepsilon &\text{if $\sqrt{\theta}\,R<r$}.
\end{cases}
\]
It is easily seen that\footnote{As a set $E$ occuring in the definition $\widetilde{D}^{s,p}(\overline{B_R}^{\,c})$ one can take for example $E=\overline B_{\sqrt{\theta}\,R}^c$.} $w\in \widetilde{D}^{s,p}(\overline{B_R}^{\,c})$. Moreover,
by using \eqref{uvgeq}, it is readily verified that $w\leq U$ in $\overline{B_R}$. We claim that, for sufficiently small $\varepsilon_0$ and $\varepsilon$ and sufficiently large $R$,  it holds
\[
(-\Delta_p)^s w\leq (-\Delta_p)^s U,\qquad \mbox{ in }\overline{B_R}^{\,c}.
\]
This would end the proof, since Theorem \ref{comparison} would yield $U\geq w$ in $\mathbb{R}^N$ and then 
\[
m=\liminf_{r\to +\infty} r^{\frac{N-s\,p}{p-1}}\,U(r)=\liminf_{r\to +\infty} \frac{U(r)}{\widetilde\Gamma(r)}\geq \liminf_{r\to +\infty} g(r)= m+\varepsilon,
\]
giving a contradiction.
The function $w-(m+\varepsilon)\,\widetilde \Gamma$ is supported in 
$B_{\sqrt{\theta}\,R}\Subset B_{R}$ and thus using Proposition \ref{nonlocalb} with
\[
\Omega=\overline{B_R}^{\,c},\qquad u=(m+\varepsilon)\,\widetilde\Gamma,\qquad f= (-\Delta_p)^s\Big((m+\varepsilon)\,\widetilde\Gamma\Big) ,\qquad v=w-(m+\varepsilon)\,\widetilde\Gamma,
\] 
and \eqref{splv}, for any $|x|>R$ it holds 
\begin{equation}
\label{uno}
\begin{split}
(-\Delta_p)^s w(x)&= (m+\varepsilon)^{p-1}\,(-\Delta_p)^s\widetilde\Gamma(x)\\
&\quad+\int_{B_{\sqrt{\theta}\,R}}\frac{J_p\big((m+\varepsilon)\,\widetilde\Gamma(x)-w(y)\big)-J_p\big((m+\varepsilon)\,(\widetilde\Gamma(x)-\widetilde\Gamma(y))\big)}{|x-y|^{N+s\,p}}\, dy\\
&\leq \frac{C}{|x|^{N+s\,p}}+\int_{B_{\sqrt{\theta}\,R}}\frac{h(x, y)}{|x-y|^{N+s\,p}}\, dy,
\end{split}
\end{equation}
where 
\[
h(x, y)=J_p\big((m+\varepsilon)\,(\widetilde\Gamma(y)-\widetilde\Gamma(x))\big)-J_p\big(w(y)-(m+\varepsilon)\,\widetilde\Gamma(x)\big).
\]
We now decompose the last integral in \eqref{uno} as follows
\begin{equation}
\label{scassamento}
\int_{B_{\sqrt{\theta}\,R}}\, dy=\int_{B_{R_0}}\, dy+\int_{B_{\theta\,R}\setminus B_{R_0}}\, dy+\int_{B_{\sqrt{\theta}\,R}\setminus B_{\theta\,R}}\, dy,
\end{equation}
and proceed to estimate each term separately.
\par
Being $R_0=R_0(\varepsilon_0)$ and $h$ universally bounded, it holds 
\begin{equation}
\label{due}
\int_{B_{R_0}}\frac{h(x, y)}{|x-y|^{N+s\,p}}\, dy\leq \|h\|_{L^\infty(\mathbb{R}^{2\,N})}\,\frac{\omega_N\,R_0^N}{\Big||x|-R_0\Big|^{N+s\,p}}\leq \frac{C(\varepsilon_0)}{|x|^{N+s\,p}}\,\left(1-\theta\right)^{-N-s\,p},
\end{equation}
where we used that (recall that we are assuming $\theta\,R> R_0$)
\[
\Big||x|-R_0\Big|\ge \left(1-\frac{R_0}{R}\right)\,|x|\ge (1-\theta)\,|x|,\qquad \mbox{ for }|x|>R.
\]
For the second integral in \eqref{scassamento}, we notice that for $y\in B_{\theta\, R}\setminus B_{ R_0}$ and $x\in B_R^c$ we have 
\[
h(x, y)=J_p\big((m+\varepsilon)\,(\widetilde\Gamma(y)-\widetilde\Gamma(x))\big)-J_p\big((m+\varepsilon)\,(\widetilde\Gamma(y)-\widetilde\Gamma(x))-(\varepsilon+\varepsilon_0)\,\widetilde\Gamma(y)\big).
\]
Observe that by \eqref{licciz}, with simple manipulations we get
\[
\begin{split}
h(x, y)&\le c\, \Big[(m+\varepsilon)^{p-2}+(\varepsilon+\varepsilon_0)^{p-2}\Big]\,(\varepsilon+\varepsilon_0)\,\widetilde \Gamma(y),
\end{split}
\]
for $x\in B_R^c$, $y\in B_{\theta R}\setminus B_{R_0}$ and $c=c(p)>0$. Therefore, since
\[
|x-y|\ge \Big||x|-|y|\Big|\ge |x|-\theta\,R\ge \left(1-\theta\right)\,|x|,\qquad \mbox{ for } x\in  B_R^c,\ y\in B_{\theta\, R},
\] 
recalling the definition of $\widetilde\Gamma$ we get
\begin{equation}
\label{tre}
\begin{split}
\int_{B_{\theta R}\setminus B_{R_0}}\frac{h(x, y)}{|x-y|^{N+s\,p}}\, dy&\leq \frac{C\,(\varepsilon+\varepsilon_0)}{(1-\theta)^{N+s\,p}\,|x|^{N+s\,p}}\int_{B_{\theta R}\setminus B_{ R_0}}\frac{1}{|y|^{N-s\,p}}\, dy\\
&\leq  C\,\frac{(\varepsilon+\varepsilon_0)}{(1-\theta)^{N+s\,p}}\,\frac{(\theta\,R)^{s\,p}}{|x|^{N+s\,p}},
\end{split}
\end{equation}
where $C=C(N,s,p,M+m)>0$.
For the third integral in \eqref{scassamento}, for $y\in B_{\sqrt{\theta}R}\setminus B_{\theta R}$ we have
\[
\begin{split}
h(x, y)&=J_p\Big((m+\varepsilon)(\widetilde\Gamma(y)-\widetilde\Gamma(x))\Big)-J_p\left((m+\varepsilon)\,(\widetilde\Gamma(y)-\widetilde\Gamma(x))+\left(\frac{M-m}{2}-\varepsilon\right)\,\widetilde\Gamma(y)\right)\\
&\leq J_p\Big(\underbrace{(m+\varepsilon)\,(\widetilde\Gamma(y)-\widetilde\Gamma(x))}_a\Big)-J_p\Big(\underbrace{(m+\varepsilon)\,(\widetilde\Gamma(y)-\widetilde\Gamma(x))}_{a}+\underbrace{\left(\frac{M-m}{4}\right)\,\widetilde\Gamma(y)}_b\Big),
\end{split}
\]
since $\varepsilon<(M-m)/4$. The inequality \eqref{27} thus gives
\[
h(x, y) \leq -2^{2-p}\,\left(\frac{M-m}{4}\right)^{p-1}\,\widetilde\Gamma(y)^{p-1}.
\]
Therefore,  using 
\[
|x-y|\leq 2\,|x|,\qquad \mbox{ for } |x|>R, \ |y|<\sqrt{\theta}\,R,
\]
we obtain
\begin{equation}
\label{quattro}
\begin{split}
\int_{B_{\sqrt{\theta}\,R}\setminus B_{\theta\, R}}\frac{h(x, y)}{|x-y|^{N+s\,p}}\, dy &\leq -\frac{c\,(M-m)^{p-1}}{|x|^{N+s\,p}}\int_{B_{\sqrt{\theta}\,R}\setminus B_{\theta\, R}} |y|^{s\,p-N}\, dy\\
&\leq -c\,\theta^\frac{s\,p}{2}\,\Big(1-\theta^\frac{s\,p}{2}\Big)\,(M-m)^{p-1}\,\frac{R^{s\,p}}{|x|^{N+s\,p}},
\end{split}
\end{equation}
for a constant $c=c(N,s,p)>0$. 
Gathering toghether the estimates \eqref{uno}, \eqref{due}, \eqref{tre} and \eqref{quattro} we proved
\[
\begin{split}
(-\Delta_p)^s w(x)&\leq \left(C+\frac{C(\varepsilon_0)}{(1-\theta)^{N+s\,p}}\right)\,\frac{1}{|x|^{N+s\,p}}\\
&\quad -\left[c\,\Big(1-\theta^\frac{s\,p}{2}\Big)\,(M-m)^{p-1}-\frac{C\,(\varepsilon+\varepsilon_0)}{(1-\theta)^{N+s\,p}}\right]\,\frac{R^{s\,p}\,\theta^{s\,p}\,}{|x|^{N+s\,p}}.
\end{split}
\]
Thus we can choose $\varepsilon+\varepsilon_0$ small enough (depending only on $N,p,s,M-m$ and the chosen minimizer $U$), so that the second term above is negative.
For any such a choice we have, for any $|x|> R$,
\[
(-\Delta_p)^s w(x)\leq \frac{C(\varepsilon_0)}{|x|^{N+s\,p}},\qquad (-\Delta_p)^s U(x)= U(x)^{p^*-1}\geq \frac{1}{C\,|x|^{N+\frac{s\,p}{p-1}}},
\]
where in the last estimate we used Corollary \ref{lower-Dec}.
Since $p> 2$, for sufficiently large $R$ it holds
\[
 \frac{1}{C\,|x|^{N+\frac{s\,p}{p-1}}}\geq \frac{R^{s\,p\,\frac{p-2}{p-1}}}{C\,|x|^{N+s\,p}}\geq \frac{C\,(\varepsilon_0)}{|x|^{N+s\,p}},
\]
and thus the claim follows.
\vskip.2cm\noindent
$\bullet$  \underline{\em Case $1<p<2$}. There exists $R_0=R_0(\varepsilon_0)>2$ such that 
\[
\frac{U(r)}{\widetilde\Gamma(r)}\leq M+\varepsilon_0,\quad \mbox{ for } r\geq R_0
\]
and we can choose an arbitrarily large $R>R_0$ such that 
$$
\frac{U(R)}{\widetilde\Gamma(R)}\leq m+\frac{M-m}{4}. 
$$
As before, we consider $\delta=(M-m)/4$ in Lemma \ref{lm:intervalli}: there exists $\theta<1$ so that for any such $R$ it holds
\begin{equation}
\label{uvleq}
\frac{U(r)}{\widetilde\Gamma(r)}\leq \frac{M+m}{2},\qquad \mbox{ for every } r\in [R, R/\theta].
\end{equation}
Since $U\in L^\infty(\mathbb{R}^N)$, there exists $\overline C>0$ such that $U\leq \overline C\, \widetilde\Gamma$ in $\mathbb{R}^N$, then for any $0<\varepsilon<(M-m)/4$ we consider the upper barrier $w(r)=g(r)\,\widetilde\Gamma(r)$, where 
\[
g(r)=
\begin{cases}
\overline C&\text{if $r<R_0$},\\
M+\varepsilon_0&\text{if $R_0\leq r<R$},\\
\frac{M+m}{2}&\text{if $R\leq r<R/\sqrt{\theta}$},\\
M-\varepsilon &\text{if $R/\sqrt{\theta}<r$}.
\end{cases}
\]
Again, it is easy to verify that $w\in \widetilde{D}^{s,p}(\overline{B_R}^{\,c})$.
Using \eqref{uvleq}, we can verify that $w\geq U$ in $B_{R/\theta}$. We claim that, for sufficiently small $\varepsilon_0$ and $\varepsilon$ and sufficiently large $R$,  it holds
\[
(-\Delta_p)^s w\geq (-\Delta_p)^s U,\qquad \mbox{ in } B^c_{R/\theta}.
\]
This would end the proof, since the comparison principle of Theorem \ref{comparison} would yield $U\leq w$ in $\mathbb{R}^N$ and then 
\[
M=\limsup_{r\to +\infty} r^{\frac{N-s\,p}{p-1}}\,U(r)=\limsup_{r\to+\infty} \frac{U(r)}{\widetilde\Gamma(r)}\leq \limsup_{r\to +\infty} g(r)= M-\varepsilon,
\]
which gives again a contradiction. 
The function $w-(M-\varepsilon)\,\widetilde\Gamma$ is supported in 
$B_{R/\sqrt{\theta}}\Subset B_{R/\theta}$ and thus using again Proposition \ref{nonlocalb} 
with
\[
\Omega=\overline{B_{R/\theta}}^{\,c},\qquad u=(M-\varepsilon)\,\widetilde \Gamma,\qquad f= (-\Delta_p)^s\Big((M-\varepsilon)\,\widetilde\Gamma\Big) ,\qquad v=w-(M-\varepsilon)\,\widetilde \Gamma,
\] 
and \eqref{splv}, for any $|x|>R/\theta$ it holds 
\begin{equation}
\begin{split}
\label{unobis}
(-\Delta_p)^s w(x)&= (M-\varepsilon)^{p-1}\,(-\Delta_p)^s\,\widetilde\Gamma(x) \\
&\quad+\int_{B_{R/\sqrt{\theta}}}\frac{J_p\big((M-\varepsilon)\,\widetilde\Gamma(x)-w(y)\big)-J_p\big((M-\varepsilon)\,(\widetilde\Gamma(x)-\widetilde\Gamma(y))\big)}{|x-y|^{N+s\,p}}\, dy \\
&\geq \frac{1}{C\,|x|^{N+sp}}+\int_{B_{R/\sqrt{\theta}}}\frac{h(x, y)}{|x-y|^{N+sp}}\, dy,
\end{split}
\end{equation}
where 
\[
h(x, y)=J_p\big((M-\varepsilon)\,(\widetilde\Gamma(y)-\widetilde\Gamma(x))\big)-J_p\big(w(y)-(M-\varepsilon)\,\widetilde\Gamma(x)\big).
\]
As above, we now decompose the last integral in \eqref{unobis} as
\[
\int_{B_{R/\sqrt{\theta}}}\, dy=\int_{B_{R_0}}\, dy+\int_{B_{R}\setminus B_{R_0}}\, dy+\int_{B_{R/\sqrt{\theta}}\setminus B_{R}}\, dy,
\]
and proceed to estimate each term separately.
\par
Being $R_0=R_0(\varepsilon_0)$ and $h$ universally bounded, as before we get 
\begin{equation}
\label{duebis}
\int_{B_{R_0}}\frac{h(x, y)}{|x-y|^{N+sp}}\, dy\geq -\frac{C(\varepsilon_0)}{|x|^{N+s\,p}},
\end{equation}
where this time we used that (recall that we are assuming $R>R_0$)
\[
\Big||x|-R_0\Big|\ge \left(1-\frac{R_0}{R}\,\theta\right)\,|x|\ge (1-\theta)\,|x|,\qquad \mbox{ for }|x|>R/\theta.
\]
For $y\in B_{R}\setminus B_{ R_0}$ we have 
\[
h(x, y)=J_p\Big((M-\varepsilon)\,(\widetilde\Gamma(y)-\widetilde\Gamma(x))\Big)-J_p\Big((M-\varepsilon)\,(\widetilde\Gamma(y)-\widetilde\Gamma(x))+(\varepsilon+\varepsilon_0)\,\widetilde\Gamma(y)\Big),
\]
and by subaddivity of $\tau\mapsto \tau^{p-1}$, we get
\[
h(x, y)\geq -(\varepsilon+\varepsilon_0)^{p-1}\,\widetilde\Gamma(y)^{p-1}.
\]
Therefore, the analogue of \eqref{tre} is now 
\begin{equation}
\label{trebis}
\int_{B_{R}\setminus B_{ R_0}}\frac{h(x, y)}{|x-y|^{N+sp}}\, dy\geq-C\,(\varepsilon+\varepsilon_0)^{p-1}\frac{R^{s\,p}}{|x|^{N+s\,p}},  
\end{equation}
and again $C=C(N,s,p,M+m)>0$. For the previous estimate we also used that
\[
|x-y|\ge |x|-R\ge \left(1-\theta\right)\,|x|,\qquad \mbox{ for } |x|>R/\theta,\ |y|<R.
\] 
For $y\in B_{R/\sqrt{\theta}}\setminus B_{R }$ and $x\in B^c_{R/\theta}$ we have
\[
\begin{split}
h(x, y)&=J_p\Big(\underbrace{(M-\varepsilon)\,(\widetilde\Gamma(y)-\widetilde\Gamma(x))}_a\Big)\\
&-J_p\Big(\underbrace{(M-\varepsilon)\,(\widetilde\Gamma(y)-\widetilde\Gamma(x))}_a-\underbrace{\left(\frac{M-m}{2}-\varepsilon\right)\,\widetilde\Gamma(y)}_b\Big).
\end{split}
\]
Clearly
\[
0\leq a=(M-\varepsilon)\,(\widetilde\Gamma(y)-\widetilde\Gamma(x))\leq (M-\varepsilon)\,\widetilde \Gamma(y)=:A,
\]
so that \eqref{ineqb} provides
\[
h(x, y)\geq \max\left\{(M-\varepsilon)^{p-1}-\left(\frac{M+m}{2}\right)^{p-1},\,\left(\frac{M-m}{2}-\varepsilon\right)^{p-1}2^{1-p}\right\} \,\widetilde\Gamma(y)^{p-1}.
\]
Proceeding as for \eqref{quattro} and using
\[
|x-y|\leq 2\,|x|,\qquad \mbox{ for } x\in B^c_{R/\theta}, \ y\in B_{R/\sqrt{\theta}},
\]
we thus obtain
\begin{equation}
\label{quattrobis}
\int_{B_{R/\sqrt{\theta}}\setminus B_{R}}\frac{h(x, y)}{|x-y|^{N+s\,p}}\, dy \ge \frac{c}{|x|^{N+s\,p}}\int_{B_{R/\sqrt{\theta}}\setminus B_{R}} |y|^{s\,p-N}\, dy \ge c\,\frac{R^{s\,p}}{|x|^{N+s\,p}},
\end{equation}
for a small constant $c$ depending only on $M$ and $m$.
Gathering together the estimates \eqref{unobis}, \eqref{duebis}, \eqref{trebis} and \eqref{quattrobis}, we proved
\[
(-\Delta_p)^s w(x)\geq -\frac{C(\varepsilon_0)}{|x|^{N+s\,p}}+\Big(c-C\,(\varepsilon+\varepsilon_0)^{p-1}\Big)\,\frac{R^{s\,p}}{|x|^{N+s\,p}}.
\]
in $B_{R/\theta}^c$. We can thus choose $\varepsilon_0$ and $\varepsilon$ small enough so that the second term above is positive. For any such choice we have, for any $|x|>  R/\theta$,
\[
(-\Delta_p)^s w(x)\geq -\frac{C(\varepsilon_0)}{|x|^{N+s\,p}}+\frac{c}{2}\frac{R^{s\,p}}{|x|^{N+s\,p}},
\]
and for sufficiently large $R$ so that $c\,R^{s\,p}>4\,C(\varepsilon_0)$ it holds
\[
(-\Delta_p)^s w(x)\geq \frac{c}{4}\frac{R^{s\,p}}{|x|^{N+s\,p}}.
\]
By using Corollary \ref{lower-Dec} and the fact that $1<p<2$, for every $|x|\ge R/\theta$ we get
\[
(-\Delta_p)^s U(x)= U^{p^*-1}(x)\leq \frac{C}{|x|^{N+\frac{s\,p}{p-1}}}\leq \frac{C\,\theta^{s\,p\,\frac{2-p}{p-1}}}{R^{s\,p\,\frac{2-p}{p-1}}|x|^{N+s\,p}}.
\]
We thus conclude that $(-\Delta_p)^sU\leq (-\Delta_p)^s w$ in $\overline{B_{R/\theta}}^{\,c}$ for $R$ sufficiently large, as desired.
\end{proof}

\appendix

\section{Power functions}

\label{sec:appA}

\noindent
We have the following result on power functions.
\begin{lemma}
\label{lm:sobolevpovero}
Let $0<(N-s\,p)/p<\beta<N/(p-1)$. For every $R>0$, the function $x\mapsto |x|^{-\beta}$ belongs to $\widetilde D^{s,p}(\overline{B_R}^{\,c})$. 
\end{lemma}
\begin{proof} 
A direct computation shows that $x\mapsto |x|^{-\beta}$ belongs to $L^{p-1}_{\mathrm{loc}}(\mathbb{R}^N)\cap L^{p^*}(B_R^c)$, 
when $\beta$ is as in the statement. We take $r<R$, then $E=\overline{B_r}^{\,c}\supset \overline{B_R}^{\,c}$ and we need to show
\begin{equation}
\label{claimnps}
\Big[|x|^{-\beta}\Big]_{W^{s,p}(B_r^c)}<+\infty,\qquad \text{for }\,\, \frac{N-s\,p}{p}<\beta.
\end{equation}
We compute in polar coordinates
\[
\begin{split}
\int_{B_r^c\times B_r^c}&\frac{||x|^{-\beta}-|y|^{-\beta}|^p}{|x-y|^{N+sp}}\, dx\, dy= \int_{\mathbf{S}^{N-1}\times \mathbf{S}^{N-1}}\int_{r}^{+\infty}\int_r^{+\infty}\frac{|\varrho^{-\beta}-t^{-\beta}|^p\,\varrho^{N-1}\,t^{N-1}}{|\varrho\,\omega_1-t\,\omega_2|^{N+sp}}\, d\varrho\, dt\, d\omega_1\, d\omega_2\\
&=2\int_{r}^{+\infty}\frac{\varrho^{-\beta\, p}\,\varrho^{2\,N-2}}{\varrho^{N+sp}}\int_r^{\varrho}\left|1-\left(\frac{t}{\varrho}\right)^{-\beta}\right|^p\int_{\mathbf{S}^{N-1}\times \mathbf{S}^{N-1}}\frac{d\omega_1\, d\omega_2}{|\omega_1-(t\,\omega_2)/\varrho|^{N+s\,p}}\left(\frac{t}{\varrho}\right)^{N-1}\, dt\, d\varrho\\
&=2\int_r^{+\infty}\frac{\varrho^{-\beta\, p}\,\varrho^{2\,N-1}}{\varrho^{N+s\,p}}\int_{r/\varrho}^1|1-\xi^{-\beta}|^p\,\xi^{N-1}\int_{\mathbf{S}^{N-1}\times \mathbf{S}^{N-1}}\frac{d\omega_1\, d\omega_2}{|\omega_1-\xi\,\omega_2|^{N+s\,p}}\, d\xi\, d\varrho.
\end{split}
\]
Let us now prove that for $0<\xi<1$ it holds
\[
\int_{\mathbf{S}^{N-1}\times \mathbf{S}^{N-1}}\frac{d\omega_1\, d\omega_2}{|\omega_1-\xi\,\omega_2|^{N+s\,p}}\leq \frac{C}{(1-\xi)^{1+s\,p}}.
\]
Without loss of generality, we may assume that $\xi\geq 1/2$, since for $0<\xi<1/2$ the integral
is uniformly bounded.
By rotational invariance, we have
\[
\int_{\mathbf{S}^{N-1}\times \mathbf{S}^{N-1}}\frac{d\omega_1\, d\omega_2}{|\omega_1-\xi\,\omega_2|^{N+s\,p}}=
|\mathbf{S}^{N-1}|\int_{\mathbf{S}^{N-1}}\frac{d\omega_2}{|{\bf e}_1-\xi\,\omega_2|^{N+s\,p}},
\]
where ${\bf e}_1=(1,0,\dots,0)$. By changing variable $\omega_2=(t,z)$ with 
$$
t=\pm\sqrt{1-|z|^2},\qquad z\in B_1'\subset \mathbb{R}^{N-1},
$$ 
we  therefore get (the constant $C$ may vary from a line to another)
\begin{align*}
\int_{\mathbf{S}^{N-1}}\frac{d\omega_2}{|{\bf e}_1-\xi\,\omega_2|^{N+s\,p}}&= \int_{\mathbf{S}^{N-1}\setminus B_1({\bf e_1})}\frac{d\omega_2}{|{\bf e}_1-\xi\,\omega_2|^{N+s\,p}}+\int_{\mathbf{S}^{N-1}\cap B_1({\bf e_1})}\frac{d\omega_2}{|{\bf e}_1-\xi\,\omega_2|^{N+s\,p}}\\
&\leq C\left(1+ \int_{B_1'}\frac{dz}{((1-\xi\, t)^2+\xi^2\,|z|^2)^{\frac{N+s\,p}{2}}}\right) \\
&\leq C\left(1+\int_{B_1'}\frac{dz}{((1-\xi)^2+\xi^2\,|z|^2)^{\frac{N+s\,p}{2}}} \right)\\
&\leq C\left(1+\frac{1}{(1-\xi)^{1+s\,p}}\int_{B_{\frac{\xi}{1-\xi}}'}\frac{1}{
(1+|y|^2)^{\frac{N+s\,p}{2}}}dy\right)\\
&\leq C\left(1+\frac{1}{(1-\xi)^{1+s\,p}}\int_{\mathbb{R}^{N-1}}\frac{1}{
(1+|y|^2)^{\frac{N+s\,p}{2}}}dy\right)
\end{align*}
which proves the claim. Taking into account that for $0< \xi <1$ it also holds
\[
\frac{|1-\xi^{-\beta}|^p}{|1-\xi|^{1+s\,p}}\leq C\,(\xi^{-\beta\, p}+|1-\xi|^{p\,(1-s)-1})
\]
we therefore get
\[
\Big[|x|^{-\beta}\Big]_{W^{s,p}(B_r^c)}^p\leq C\int_r^{+\infty}\varrho^{N-1-p\,(s+\beta)}\,d\varrho\int_{r/\varrho}^1  \xi^{N-1}\,\big(\xi^{-\beta\, p}+|1-\xi|^{p\,(1-s)-1}\big)\, d\xi.
\]
All the integrals are now explicitly computable and one can readily get \eqref{claimnps}.
\end{proof}

\begin{lemma}
\label{power-lemma}
Let $0<(N-s\,p)/p<\beta<N/(p-1)$. For every $R>0$, it holds
\[
(-\Delta_p)^s |x|^{-\beta}=C(\beta)\, |x|^{-\beta\,(p-1)-s\,p}\quad \mbox{ weakly in } \overline{B_R}^{\,c},
\]
where the constant $C(\beta)$ is given by
\begin{equation}
\label{formulaC}
C(\beta)=2\,\int_0^1 \varrho^{s\,p-1}\,\left[1-\varrho^{N-s\,p-\beta\,(p-1)}\right]\,\left|1-\varrho^{\beta}\right|^{p-1}\, \Phi(\varrho)\,d\varrho,
\end{equation}
and
\begin{equation}
\label{fihona}
\Phi(\varrho)=\mathcal{H}^{N-2}(\mathbf{S}^{N-2})\,\int_{-1}^1 \frac{(1-t^2)^\frac{N-3}{2}}{\Big(1-2\,t\,\varrho+\varrho^2\Big)^\frac{N+s\,p}{2}}\, dt.
\end{equation}
\end{lemma}
\begin{proof}
Observe that
$$
|x|^{-\beta\,(p-1)-s\,p} \in L^{(p*)'}(B_R^c), \,\,\quad\text{for any $\beta>(N-s\,p)/p$}.
$$
Then, by Theorem \ref{density} and Proposition \ref{deltaps} it suffices to show that
\[
\int_{\mathbb{R}^{2N}}\frac{J_p(|x|^{-\beta}-|y|^{-\beta})}{|x-y|^{N+s\,p}}\,(\varphi(x)-\varphi(y))\,dx\,dy=C(\beta)\,\int_\Omega |x|^{-\beta\,(p-1)-s\,p}\,\varphi\,dx,
\]
for an arbitrary $\varphi\in C^{\infty}_c(\overline{B_R}^{\,c})$. For every such a $\varphi$ we consider the double integral
\[
\begin{split}
\int_{\mathbb{R}^{2N}}\frac{J_p(|x|^{-\beta}-|y|^{-\beta})}{|x-y|^{N+s\,p}}\,(\varphi(x)-\varphi(y))\,dx\,dy.
\end{split}
\]
We observe that this is absolutely convergent, indeed
\[
\begin{split}
\int_{\mathbb{R}^{2N}}&\frac{|J_p(|x|^{-\beta}-|y|^{-\beta})|}{|x-y|^{N+s\,p}}\,|\varphi(x)-\varphi(y)|\,dx\,dy\\
&=\int_{B_R^c\times B_R^c}\frac{|J_p(|x|^{-\beta}-|y|^{-\beta})|}{|x-y|^{N+s\,p}}\,|\varphi(x)-\varphi(y)|\,dx\,dy\\
&\quad +2\,\int_{B_R} \int_{\mathrm{supp}(\varphi)} \frac{|J_p(|x|^{-\beta}-|y|^{-\beta})|}{|x-y|^{N+s\,p}}\,|\varphi(y)|\,dx\,dy\\
&\le \Big[|x|^{-\beta}\Big]_{W^{s,p}(B_R^c)}\,[\varphi]_{W^{s,p}(B_R^c)}+C\,\|\varphi\|_{L^\infty}\,|\mathrm{supp}(\varphi)|\,\int_{B_R} |x|^{-\beta\,(p-1)}\,dx,
\end{split}
\]
and both terms are finite, thanks to Lemma \ref{lm:sobolevpovero}.
For $\delta>0$ we consider the conical set
\[
\mathcal{O}_\delta=\{(x,y)\in\mathbb{R}^{2\,N}\, :\, (1-\delta)\,|x|\le |y|\le (1+\delta)\,|x|\},
\]
then by the Dominated Convergence Theorem
\[
\begin{split}
\lim_{\delta\searrow 0}\int_{\mathcal{O}_\delta^c}& \frac{J_p(|x|^{-\beta}-|y|^{-\beta})}{|x-y|^{N+s\,p}}\,(\varphi(x)-\varphi(y))\,dy\,dx\\
&=\int_{\mathbb{R}^{2N}}\frac{J_p(|x|^{-\beta}-|y|^{-\beta})}{|x-y|^{N+s\,p}}\,(\varphi(x)-\varphi(y))\,dx\,dy.
\end{split}
\]
We now observe that
\[
\begin{split}
\int_{\mathcal{O}_\delta^c}& \frac{J_p(|x|^{-\beta}-|y|^{-\beta})}{|x-y|^{N+s\,p}}\,(\varphi(x)-\varphi(y))\,dy\,dx=2\,\int_{\mathbb{R}^N}\left(\int_{\mathcal{K}_\delta(x)^c} \frac{J_p(|x|^{-\beta}-|y|^{-\beta})}{|x-y|^{N+s\,p}}\,dy\right)\,\varphi(x)\,dx,
\end{split}
\]
where for every $x\in\mathbb{R}^N$
\[
\mathcal{K}_\delta(x)=\{y\in\mathbb{R}^N\, :\, (1-\delta)\,|x|\le |y|\le (1+\delta)\,|x|\},
\]
and of course $\mathcal{K}_\delta(x)=\mathcal{K}_\delta(x')$ whenever $|x|=|x'|$.
We set
\[
f_\delta(x)=2\,\int_{\mathcal{K}_\delta(x)^c} \frac{J_p(|x|^{-\beta}-|y|^{-\beta})}{|x-y|^{N+s\,p}}\,dy,\qquad x\in\mathbb{R}^N\setminus\{0\},
\]
it is easily seen that $f_\delta$ is a radial function, homogeneous of degree $-\beta\,(p-1)-s\,p$ (see \cite[Lemma 6.2]{Brasco}). Thus for $x\not=0$ we have
\begin{equation}
\label{pseudohomo}
f_\delta(x)=|x|^{-\beta\,(p-1)-s\,p}\,f_{\delta}(\omega),\qquad \mbox{ for }\quad \omega=\frac{x}{|x|}\in\mathbf{S}^{N-1}.
\end{equation}
	We set
	\[
	C(\beta;\delta):=f_{\delta}(\omega)=2\,\int_{\mathcal{K}_\delta(\omega)^c} \frac{J_p(1-|y|^{-\beta})}{|\omega-y|^{N+s\,p}}\,dy,\qquad \omega\in\mathbf{S}^{N-1},
	\]
	which is independent of the direction $\omega$, by radiality of $f_\delta$. By taking the average over $\mathbf{S}^{N-1}$ and
	proceeding as in \cite[Lemma B.2]{Brasco}, we get
	\[
	C(\beta;\delta)=2\,\int_{|\varrho-1|\ge \delta} \varrho^{N-1}\,|1-\varrho^{-\beta}|^{p-2}\,(1-\varrho^{-\beta})\,\Phi(\varrho)\,d\varrho,
	\]
where $\Phi$ is defined in \eqref{fihona}. We now decompose the integral defining $C(\beta;\delta)$ and perform a change of variables, i.e.
	\[
	\begin{split}
	C(\beta;\delta)&=-2\,\int_0^{1-\delta} \varrho^{N-1}\,|1-\varrho^{-\beta}|^{p-1}\,\Phi(\varrho)\,d\varrho+2\,\int_{1+\delta}^\infty \varrho^{N-1}\,|1-\varrho^{-\beta}|^{p-1}\,\Phi(\varrho)\,d\varrho\\
	&=-2\,\int_0^{1-\delta} \varrho^{N-1-\beta\,(p-1)}\,|\varrho^{\beta}-1|^{p-1}\,\Phi(\varrho)\,d\varrho\\
	&\quad +2\,\int_0^{1/(1+\delta)} \varrho^{-N-1}\,|1-\varrho^{\beta}|^{p-1}\,\Phi(1/\varrho)\,d\varrho.
	\end{split}
	\]
Finally, observe that
\[
\Phi(1/\varrho)=\varrho^{N+s\,p}\,\Phi(\varrho),
\]
thus the quantity $C(\beta;\delta)$ can be written as
\begin{equation}
\label{bikini}
\begin{split}
C(\beta;\delta)&=2\,\int_0^{1-\delta} \left(1-\varrho^{N-s\,p-\beta\,(p-1)}\right)\,\varrho^{s\,p-1}\,(1-\varrho^{\beta})^{p-1}\,\Phi(\varrho)\,d\varrho\\
	&\quad +2\,\int_{1-\delta}^{1/(1+\delta)} \varrho^{s\,p-1}\,(1-\varrho^{\beta})^{p-1}\,\Phi(\varrho)\,d\varrho.
\end{split}
\end{equation}
Recall that $\varphi$ is compactly supported in $\overline{B_R}^{\,c}$, thus by using \eqref{pseudohomo} we can estimate
\[
\begin{split}
\left|\int_{\Omega} f_\delta\,\varphi\,dx-C(\beta)\,\int_\Omega |x|^{-\beta\,(p-1)-s\,p}\,\varphi\,dx\right|&\le \|\varphi\|_\infty\,R^{-\beta\,(p-1)-s\,p}\,|\mathrm{supp}(\varphi)|\,\Big|C(\beta;\delta)-C(\beta)\Big|.
\end{split}
\]
In order to prove that $C(\beta;\delta)$ converges to $C(\beta)$ as $\delta$ goes to $0$, we decompose the function $\Phi$ defined in \eqref{fihona} as follows
\[
\begin{split}
\Phi(\varrho)&=\int_{-1}^{1/2} \frac{(1-t^2)^\frac{N-3}{2}}{(1-2\,t\,\varrho+\varrho^2)^\frac{N+s\,p}{2}}\, dt+\int_{1/2}^{1} \frac{(1-t^2)^\frac{N-3}{2}}{(1-2\,t\,\varrho+\varrho^2)^\frac{N+s\,p}{2}}\, dt=:\Phi_1(\varrho)+\Phi_2(\varrho),
\end{split}
\]
where we omitted the dimensional constant $\mathcal{H}^{N-2}(\mathbf{S}^{N-2})$ for simplicity. Let us start estimating $\Phi_1$. If we use that 
\[
1-2\,t\,\varrho+\varrho^2=(\varrho-t)^2+(1-t^2)\ge \frac{3}{4},\qquad \mbox{ if }-1\le t\le\frac{1}{2},
\]
we get
\begin{equation}
\label{phi1}
0\le\Phi_1(\varrho)\le C,\qquad 0<\varrho<1.
\end{equation}
We now consider $\Phi_2(\varrho)$, discussing separately the cases $0<\varrho<1/2$ and $1/2\le \varrho <1$. We observe that for $0<\varrho<1/2$ we have
\[
1-2\,t\,\varrho+\varrho^2=(1-\varrho)^2+2\,\varrho\,(1-t)\ge \frac{1}{4},\qquad \mbox{ if } \frac{1}{2}\le t\le 1.
\]
Then we get again
\begin{equation}
\label{phi2}
0\le\Phi_2(\varrho)\le C,\qquad \mbox{ if } 0<\varrho<\frac{1}{2}.
\end{equation}
We are left with the term $\Phi_2(\varrho)$ for $1/2\le \varrho<1$. With simple manipulations\footnote{We use the change of variables
\(
\tau=\frac{2\,\varrho}{(1-\varrho)^2}\,(1-t).
\)
} 
we can write it as
\[
\Phi_2(\varrho)=\frac{(2\,\varrho)^{-\frac{N-1}{2}}}{(1-\varrho)^{1+s\,p}}\int_{0}^{\frac{\varrho}{(1-\varrho)^2}} \frac{\left(2-\frac{(1-\varrho)^2}{2\,\varrho}\,\tau\right)^\frac{N-3}{2}\,\tau^\frac{N-3}{2}}{(1+\tau)^\frac{N+s\,p}{2}}\, d\tau.
\]
In particular, we get
\begin{equation}
\label{phi22}
0\le\Phi_2(\varrho)\le C\,(1-\varrho)^{-1-s\,p},\qquad \mbox{ if }\frac{1}{2}\le\varrho<1.
\end{equation}
By using \eqref{phi1}, \eqref{phi2} and \eqref{phi22}, we thus obtain for the first integral in \eqref{bikini}
\[
\lim_{\delta\searrow 0}2\,\int_0^{1-\delta} \left(1-\varrho^{N-s\,p-\beta\,(p-1)}\right)\,\varrho^{s\,p-1}\,(1-\varrho^{\beta})^{p-1}\,\Phi(\varrho)\,d\varrho=C(\beta),
\]
and observe that the latter is finite, thanks to \eqref{phi22}. It is only left to show that the other integral in \eqref{bikini} converges to $0$. Still by \eqref{phi1} and \eqref{phi22}, we obtain
\[
\begin{split}
\lim_{\delta\searrow 0}\int_{1-\delta}^{1/(1+\delta)} &\varrho^{s\,p-1}\,(1-\varrho^{\beta})^{p-1}\,\Phi(\varrho)\,d\varrho\\
&\le C\,\lim_{\delta\searrow 0} \int_{1-\delta}^{1/(1+\delta)} \varrho^{s\,p-1}\,(1-\varrho^{\beta})^{p-1}\,d\varrho\\
&+C\,\lim_{\delta\searrow 0} \int_{1-\delta}^{1/(1+\delta)} \varrho^{s\,p-1}\,(1-\varrho^{\beta})^{p-1}\,(1-\varrho)^{-1-s\,p}\,d\varrho\\
&\le C\,\lim_{\delta\searrow 0} \int_{1-\delta}^{1/(1+\delta)} (1-\varrho)^{p-2-s\,p}\,d\varrho\\
&= \frac{C}{p-1-s\,p}\,\lim_{\delta\searrow 0}\,\left[-\left(\frac{\delta}{1+\delta}\right)^{p-1-s\,p}+\delta^{p-1-s\,p}\right],
\end{split}
\]
where we assumed for simplicity that $p-1-s\,p\not=0$.
If $p-1-s\,p>0$, the last term converges to $0$. If $p-1-s\,p<0$, we have
\[
\begin{split}
\left(\frac{\delta}{1+\delta}\right)^{p-1-s\,p}-\delta^{p-1-s\,p}&= \delta^{p-1-s\,p}\,\Big[(1+\delta)^{s\,p+1-p}-1\Big]\\
&\simeq (s\,p+1-p)\,\delta^{p-s\,p},\qquad \mbox{ as } \delta \searrow 0,
\end{split}
\]
and thus the integral converges to $0$ again. Finally, the borderline case $p-1-s\,p=0$ is treated similarly, we leave the details to the reader.
\par
In conclusion, we get
\[
\lim_{\delta\searrow 0}\int_{\Omega} f_\delta\,\varphi\,dx=C(\beta)\,\int_\Omega |x|^{-\beta\,(p-1)-s\,p}\,\varphi\,dx,
\]
as desired.
\end{proof}
\begin{remark}
The previous result was proved in \cite[Lemma 3.1]{FS} for the limit case $\beta=(N-s\,p)/p$. Our argument is different, since we rely on elementary estimates for the function $\Phi$, rather than on special properties of hypergeometric and beta functions like in \cite{FS}. 
\end{remark}
\noindent
Observe that the choice $\beta=(N-s\,p)/(p-1)$ is feasible in the previous results, since
\[
\frac{N-s\,p}{p}<\frac{N-s\,p}{p-1}<\frac{N}{p-1}.
\]
Moreover, with such a choice we have $C(\beta)=0$ in \eqref{formulaC}.
Then from Lemmas \ref{lm:sobolevpovero} and \ref{power-lemma},
we get the following.

\begin{theorem}
\label{fundamental}
	For any $R>0$, $\Gamma(x)=|x|^{-\frac{N-s\,p}{p-1}}$ belongs to $\widetilde{D}^{s,p}(\overline{B_R}^{\,c})$ and weakly solves $(-\Delta_p)^su=0$ in  $\overline{B_R}^{\,c}$.
\end{theorem}


\vskip15pt

\begin{thebibliography}{99}


\bibitem{AlLi}
{\sc F. J. Almgren, E. Lieb}, 
Symmetric decreasing rearrangement is sometimes continuous, 
{\em J. Amer. Math. Soc.}, {\bf 2} (1989), 683--773.

\bibitem{AMRT} 
{\sc F. Andreu, J. M. Maz\'on, J. D. Rossi, J. Toledo}, A nonlocal $p-$Laplacian evolution equation with nonhomogeneous Dirichlet boundary conditions, {\em SIAM J. Math. Anal.}, {\bf 40} (2009), 1815-1851.

\bibitem{Au} 
{\sc T. Aubin}, 
Probl\`emes isop\'erim\'etriques et espaces de Sobolev, 
{\em J. Differential Geometry}, {\bf 11} (1976), 573--598.

\bibitem{Brasco}
{\sc L.\ Brasco, G.\ Franzina}, 
Convexity properties of Dirichlet integrals and Picone-type inequalities, 
{\em Kodai Math. J.}, {\bf 37} (2014), 769--799.

\bibitem{BP} 
{\sc L. Brasco, E. Parini}, The second eigenvalue of the fractional $p-$Laplacian, 
{\em Adv. Calc.\ Var.}, to appear.

\bibitem{BraLinPar} 
{\sc L. Brasco, E. Lindgren, E. Parini}, 
The fractional Cheeger problem, 
{\em Interfaces Free Bound.}, \textbf{16} (2014), 419--458.

\bibitem{CLO}
{\sc W.\ Chen, C.\ Li,  B.\ Ou},  
Classification of solutions for an integral equation,
{\em Comm.\ Pure Appl.\ Math.}, {\bf 59} (2006), 330--343.

\bibitem{CENV} {\sc D. Cordero-Erausquin, B. Nazaret, C. Villani}, A mass-transportation approach to sharp Sobolev and Gagliardo-Nirenberg inequalities, {\em Adv. Math.}, {\bf 182} (2004), 307--332.

\bibitem{Cotsiolis}
{\sc A.\  Cotsiolis, N. Tavoularis},
Best constants for Sobolev inequalities for higher order fractional derivatives,
{\em J. Math. Anal. Appl.}, {\bf 295} (2004), 225--236.

\bibitem{ClassDino}
{\sc L.\ Damascelli, S.\ Merch\'an, L. Montoro, B. Sciunzi},
Radial symmetry and applications for a problem involving the
$-\Delta_p(\cdot)$ operator and critical nonlinerity in $\mathbb{R}^N$,
{\em Adv.\ Math.}, {\bf 265} (2014), 313--335.

\bibitem{DKP2}
{\sc A.\ Di Castro, T.\ Kuusi, G.\ Palatucci},
Nonlocal Harnack inequalities,
{\em J. Funct. Anal.}, {\bf 267} (2014), 1807--1836.

\bibitem{Guida}
{\sc E. Di Nezza, G. Palatucci, E. Valdinoci}, Hitchhiker's guide to the fractional Sobolev spaces, {\em Bull. Sci. Math.}, {\bf 136} (2012), 521--573.

\bibitem{FSV}
{\sc A.\ Fiscella, R.\ Servadei, E.\ Valdinoci},
Density properties for fractional Sobolev spaces,
{\em Ann. Acad. Sci. Fenn. Math.}, {\bf 40} (2015), 235--253.

\bibitem{federer}
{\sc H.\ Federer, W. H.\ Fleming}, 
Normal and integral currents, {\em Ann. Math.}, {\bf 72} (1960) 458--520.

\bibitem{FS} 
{\sc R. L. Frank, R. Seiringer}, 
Non-linear ground state representations and sharp Hardy inequalities, 
{\em J. Funct. Anal.}, {\bf 255} (2008), 3407--3430.

\bibitem{GM} 
{\sc N. Gigli, S. Mosconi}, 
The abstract Lewy-Stampacchia inequality and applications, 
{\em J. Math. Pures Appl.}, {\bf 104} (2015), 258--275. 

\bibitem{GueddaV}
{\sc M.\ Guedda, L. Veron}, Local and global properties of solutions of quasilinear elliptic equations,
{\em J.\ Differential Equations}, {\bf 76} (1988), 159--189.

\bibitem{ILPS}
{\sc A. Iannizzotto, S. Liu, K. Perera, M. Squassina}, 
Existence results for fractional $p-$Laplacian problems via Morse theory,
{\em Adv. Calc. Var.} (2015), to appear.

\bibitem{IMS}
\textsc{A.\ Iannizzotto, S.\ Mosconi, M.\ Squassina},
Global H\"older regularity for the fractional $p$-Laplacian,
to appear on {\em Rev. Mat. Iberoamericana}, available at {\tt http://arxiv.org/abs/1411:2956}

\bibitem{IMS2}
\textsc{A.\ Iannizzotto, S.\ Mosconi, M.\ Squassina},
A note on global regularity for the weak solutions of fractional $p-$Laplacian equations, to appear on {\em Atti Accad. Naz. Lincei Rend. Lincei Mat. Appl.}, available at {\tt http://arxiv.org/abs/1504.01006}

\bibitem{IN} 
\textsc{H. Ishii, G. Nakamura}, A class of integral equations and approximation of $p-$Laplace equations, {\em Calc. Var. Partial Differential Equations}, {\bf 37}
(2010), 485--522.

\bibitem{KMS1} 
{\sc T. Kuusi, G. Mingione, Y. Sire}, 
Nonlocal equations with measure data, 
{\em Comm. Math. Phys.}, {\bf 337} (2015), 1317--1368.

\bibitem{lieb}
{\sc E. H. Lieb}, Sharp constants in the Hardy-Littlewood-Sobolev and related inequalities, {\em Ann. Math.}, {\bf 118} (1983), 349--374.

\bibitem{liebL}
{\sc E. H. Lieb, M. Loss}, Analysis, Second edition. Graduate Studies in Mathematics, {\bf 14}. American Mathematical Society, Providence, RI, 2001.

\bibitem{Lions}
{\sc P. L. Lions}, 
The concentration-compactness principle in the calculus of variations. 
The limit case. I, {\em Rev. Mat. Iberoamericana}, {\bf 1} (1985), 145--201.

\bibitem{LL}
{\sc E. Lindgren, P. Lindqvist}, Fractional eigenvalues, 
{\em Calc. Var. Partial Differential Equations}, {\bf 49} (2014), 795--826.

\bibitem{mazia}
{\sc V. G. Maz'ja}, 
Classes of regions and imbedding theorems for function spaces, 
{\em Soviet Math. Dokl.}, {\bf 1} (1960), 882--885.

\bibitem{MS} 
{\sc V. Maz'ya, T. Shaposhnikova}, On the Bourgain, Brezis, and Mironescu theorem concerning limiting embeddings of fractional Sobolev spaces, {\em J. Funct. Anal.}, {\bf 195} (2002), 230--238.

\bibitem{per-mosc-squ-yang}
{\sc S. Mosconi, K. Perera, M. Squassina, Y. Yang}, 
On the Brezis-Nirenberg problem for the fractional
$p-$Laplacian, preprint (2015), available at {\tt http://arxiv.org/abs/1508.00700}

\bibitem{Ro} 
{\sc G. Rosen}, Minimal value for $c$ in the Sobolev inequality, 
{\em SIAM J. Appl. Math.}, {\bf 21} (1971), 30--33.

\bibitem{SV}
{\sc R. Servadei, E. Valdinoci}, The Brezis-Nirenberg result for the fractional Laplacian, {\em Trans. Amer. Math. Soc.}, {\bf 367} (2015), 67--102.

\bibitem{sciunziCl}
{\sc B.\ Sciunzi}, 
Classification of positive ${\mathcal D}^{1,p}(\mathbb{R}^N)-$solutions to the critical $p-$Laplace equation in $\mathbb{R}^N$, preprint (2015), available at {\tt http://arxiv.org/abs/1506.03653}

\bibitem{talenti}
{\sc G.\ Talenti},
Best constant in Sobolev inequality,
{\em Ann. Mat. Pura Appl.}, {\bf 110} (1976), 353--372.

\bibitem{Vetois}
{\sc J. Vetois}, A priori estimates and application to the symmetry of solutions for critical $p-$Laplace equations,
to appear on {\em J. Differential Equations} (2015), {\tt doi:10.1016/j.jde.2015.08.041}

\end{thebibliography}
\end{document}